\documentclass[11pt]{amsart}
\usepackage[colorlinks=true, pdfstartview=FitV, linkcolor=blue, 
citecolor=blue]{hyperref}

\usepackage{amsmath,amssymb,bbm,amscd}
\usepackage{graphicx}
\usepackage{a4wide}

\theoremstyle{plain}
\newtheorem{theorem}{Theorem}[section]
\newtheorem{prop}[theorem]{Proposition}
\newtheorem{lemma}[theorem]{Lemma}
\newtheorem{coro}[theorem]{Corollary}

\theoremstyle{definition}
\newtheorem{example}[theorem]{Example}
\newtheorem{remark}[theorem]{Remark}
\newtheorem{definition}[theorem]{Definition}

\numberwithin{equation}{section}
 
\newcommand{\dd}{\,\mathrm{d}}
\newcommand{\ts}{\hspace{0.5pt}}
\newcommand{\nts}{\hspace{-0.5pt}}

\DeclareMathOperator{\dens}{\mathrm{dens}}
\DeclareMathOperator{\card}{\mathrm{card}}

\DeclareMathOperator{\supp}{\mathrm{supp}}

\newcommand{\vD}{\varDelta}

\newcommand{\vL}{\varLambda}

\newcommand{\cF}{\mathcal{F}}
\newcommand{\cG}{\mathcal{G}}

\newcommand{\ZZ}{\mathbb{Z}\ts}
\newcommand{\RR}{\mathbb{R}\ts}
\newcommand{\NN}{\mathbb{N}}

\newcommand{\XX}{\mathbb{X}}
\newcommand{\YY}{\mathbb{Y}}

\newcommand{\one}{\mathbbm{1}}

\newcommand{\defeq}{\mathrel{\mathop:}=}

\newcommand{\exend}{\hfill $\Diamond$}

\newcommand{\suc}{\mathrm{succ}}
\newcommand{\cp}{\mathrm{cp}}
\newcommand{\Sep}{\mathrm{Sep\ts}}
\newcommand{\Span}{\mathrm{Span}\ts}
\newcommand{\coin}{\mathrm{coin}}
\newcommand{\disc}{\mathrm{disc}}
\newcommand{\dc}{\mathrm{dc}}
\newcommand{\bd}{\mathrm{bd}}
\newcommand{\oo}{\mathrm{o}}
\newcommand{\eL}{\mathrm{L}}

\newcommand{\OSD}{\mathrm{OSD}}
\newcommand{\er}{R_{\mathrm{ev}}}
\newcommand{\trv}{t_{\mathrm{r}}}
\newcommand{\rc}{r^{}_{\nts\mathrm{c}}}

\begin{document}

\title[Orbit separation dimension for primitive inflation tilings]
  {Orbit separation dimension as complexity measure
    \\[2mm] for primitive inflation tilings}

\author{Michael Baake}

\author{Franz G\"{a}hler}

\address{Fakult\"{a}t f\"{u}r Mathematik,
  Universit\"{a}t Bielefeld,\newline \hspace*{\parindent}Postfach
  100131, 33501 Bielefeld, Germany}
\email{$\{$mbaake,gaehler$\}$@math.uni-bielefeld.de}

\author{Philipp Gohlke}

\address{Lund University, Centre for Mathematical Sciences,\newline
  \hspace*{\parindent}Box 118, 221 00 Lund, Sweden\newline\smallskip
  \hspace*{\parindent}{\rm and}\newline
  \hspace*{\parindent}Institut f\"ur Mathematik, Universit\"at Jena,\newline
  \hspace*{\parindent}Ernst-Abbe-Platz 1--2, 07743 Jena, Germany}
\email{philipp.gohlke@uni-jena.de}

\begin{abstract}
  Orbit separation dimension ($\OSD$), previously introduced as
  amorphic complexity, is a powerful complexity measure for
  topological dynamical systems with pure-point spectrum. Here, we
  develop methods and tools for it that allow a systematic application
  to translation dynamical systems of tiling spaces that are generated
  by primitive inflation rules. These systems share many nice
  properties that permit the explicit computation of the $\OSD$, thus
  providing a rich class of examples with non-trivial $\OSD$.
\end{abstract}


\maketitle

\medskip

\section{Introduction}\label{sec:intro}

Amorphic complexity was introduced in \cite{FGJ16,FG20,FGJK23} as a
complexity measure for topological dynamical systems $(\XX,G)$, where
$\XX$ is a compact space with some metric $d$, and $G$ a locally
compact, $\sigma$-compact Abelian group acting on $\XX$ by
homeomorphisms. The intention was to provide a finer classification
tool for the abundance of dynamical systems with vanishing
entropy. Following \cite{FGJ16,FG20,FGJK23}, see also
\cite{Henk,Karl}, for any $\delta > 0$ and $x,y \in \XX$, we define
the set
\begin{equation}\label{eq:Deltadelta}
  \vD^{}_{\delta}(x,y) \, = \,
  \{t \in G\, : \, d(x+t,y+t) \geqslant \delta \} \ts ,
\end{equation}
and use $D^{}_{\nts\delta}(x,y)$ to denote the (upper)
\emph{asymptotic density} of the set $\vD^{}_{\delta}(x,y)$, measured
with respect to some (arbitrary, but fixed) F{\o}lner sequence
$\cF = \{ {F}_n \}_{n\in\NN}$,
\begin{equation}\label{eq:Ddelta}
  D^{}_{\nts\delta}(x,y) \, = \, \dens \bigl( \vD^{}_{\delta}(x,y) \bigr)
  \, = \, \limsup_{n \rightarrow \infty}
  \frac{|\vD^{}_{\delta}(x,y) \cap {F}_n|}{|{F}_n|} \ts . 
\end{equation}
Here, $| \cdot |$ denotes the Haar measure of $G$, sometimes also
denoted by $\theta^{}_{G}$. Then, a set $S \subseteq \XX$ is
$(\delta,\nu)$-\emph{separated} if, for all $x,y \in S$ with
$x \ne y$, we have $D^{}_{\nts\delta}(x,y) \geqslant \nu$. Denoting by
$\Sep(\XX,\delta,\nu)$ the maximal cardinality of any
$(\delta,\nu)$-separated subset of $\XX$, the (upper) \emph{amorphic
  complexity} was defined as
\begin{equation}\label{eq:ac}
  \text{ac}(\XX,G) \, \defeq \, \sup_{\delta>0} \,
  \limsup_{\nu {\scriptscriptstyle\searrow} 0}
  \frac{\log \bigl( \Sep(\XX,\delta,\nu)\bigr)}{-\log(\nu)} \ts . 
\end{equation}
This is analogous to a box counting dimension, if
$\Sep ( \XX, \delta, \nu )$ is regarded as the maximum number of
disjoint open balls of diameter $\nu$ that one can squeeze into $\XX$.
Under reasonable assumptions on $G$, it turned out that
$\Sep(\XX,\delta,\nu)$, and hence $\text{ac}(\XX,G)$, can only be
finite if $(\XX,G)$ has vanishing topological entropy \cite{FGJK23}.

In view of the general notions in dynamics, including the discussions
in \cite{Karl,Henk}, we stress the following.  Amorphic (or amorphous)
systems are usually understood to have positive entropy, both in
mathematics and in physics.  In a sense, amorphic systems and systems
for which amorphic complexity is meaningful (finite) are at opposite
ends of the complexity spectrum.  In fact, for the important case
  of a minimal dynamical system $(\XX,G)$, the separation numbers
  $\Sep(\XX,\delta,\nu)$ are finite precisely when $(\XX,G)$ has
  \emph{topological} pure-point spectrum, that is, if it is uniquely
  ergodic and has pure point dynamical spectrum with continuous
  eigenfunctions \cite{FGJK23}.  Hence, systems with pure-point
spectrum are in many respects much more like (aperiodic) crystals, not
like amorphic systems. Calling a complexity measure that is meaningful
only for such systems \emph{amorphic} seems to be misleading. In view
of its definition and the fact that this quantity has the flavour of a
dimension, we therefore propose to call it \emph{orbit separation
  dimension} ($\OSD$) in the rest of this paper, and denote it by
$\OSD(\XX,G)$.

The concept itself, however, is an interesting and powerful tool,
which allows, both theoretically and practically, to analyse
\emph{topological dynamical systems} (TDS).  In the context of
symbolic dynamical systems, it has been studied for constant-length
substitutions in \cite{FG20}, while several classes of model sets were
covered in \cite{FGJK23}. In this paper, we apply it to spaces $\XX$
of self-similar aperiodic tilings of $\RR^d$ that are generated by a
primitive inflation rule. Such inflation tilings constitute a rich
class of interesting examples whose translation dynamical systems are
automatically strictly ergodic, and share other nice properties.  We
expect that some transfer will also be possible to systems of
self-affine tilings, which are generally more complicated and will
certainly need further investigations.

Due to the self-similarity, tools are available \cite{Sol97} to
determine whether the spectrum of the translation action is pure-point
--- a necessary requirement for finite $\OSD$. If that is the case,
these systems are almost $1$-to-$1$ extensions of their underlying
\emph{maximal equicontinuous factor} (MEF), and all their
eigenfunctions can be chosen continuous \cite{Sol07} (for a discussion
of the MEF in our context, see \cite{BK13}). In a similar way, the
self-similarity can also be used to determine the $\OSD$ exactly in
many cases, which provides an invariant under topological conjugation
for distinguishing tiling spaces which are otherwise difficult to tell
apart. The class of primitive inflation tilings is therefore an ideal
play ground to study and apply $\OSD$, where it is both non-trivial
and practically computable.  \vspace*{1mm}

The remainder of this paper is organised as follows. Before we start
with the main topic, Section~\ref{sec:prelim} collects some facts and
properties of mean equicontinuity and primitive inflation tilings.  In
Section~\ref{sec:trans}, we then construct a transversal to the
continuous group action. The reason is that this transversal already
determines the $\OSD$, and can be constructed as the fixed point of a
graph-directed iterated function system (GIFS). The latter then allows
us to compute the $\OSD$ in terms of a discrepancy inflation in
Section~\ref{sec:iter}.

Before we illustrate, in Section~\ref{sec:ex}, the concepts developed
in this paper with an extensive collection of examples, we take a closer
look at product tilings in Section~\ref{sec:prod}, and briefly comment
on the relation between $\OSD$ and the dimension of the window boundary
of cut-and-project sets in Section~\ref{sec:win}. We end the paper with
some concluding remarks in Section~\ref{sec:conc}.

\clearpage

\section{Preliminaries}\label{sec:prelim}

Let us collect some results on mean equicontinuity of a dynamical
system, and its relation to $\OSD$; see \cite{FGL21,FG20,FGJK23} for
details. We do not deal with the most general case here, but constrain
ourselves to $G$ being Abelian, which simplifies things considerably.
Also, in view of our tiling applications, we assume the dynamical
system to be strictly ergodic.

So, we have a compact metric space $\XX$ with metric $d$, and a
locally compact, second countable Abelian group $G$ which acts
continuously on $\XX$ by homeomorphisms, thereby defining a TDS
$(\XX,G)$. We write the action of $G$ additively. A \emph{F{\o}lner
  sequence} $\cF = (F_n)_{n\in\NN}$ is a sequence of non-empty compact
subsets $F_n$ of $G$ such that, for all $g \in G$,
\[
  \lim_{n\to\infty} \frac{|(F_n+g)\triangle F_n|}{|F_n|}
  \, = \, 0 \ts , 
\]
where $\triangle$ refers to the symmetric difference. Note that, for
$G$ Abelian, there is no need to distinguish between left and right
F{\o}lner sequences. Given a F{\o}lner sequence $\cF$, we say that
$(\XX,G)$ is (Besicovitch) $\cF$-\emph{mean equicontinuous} if, for
all $\varepsilon>0$, there exists $\delta_\varepsilon>0$ such that
\begin{equation}\label{eq:DF}
  D^{}_{\nts\cF}(x,y) \, \defeq \, \limsup_{n\rightarrow\infty}
  \frac{1}{|F_n|}\int_{F_n} d(x+t,y+t){\dd}\theta^{}_{G}(t)
  \, < \, \varepsilon
\end{equation}
for all $x,y\in\XX$ with $d(x,y)<\delta_\varepsilon$, where the
integration is with respect to the Haar measure of $G$. Note that
$D^{}_{\nts\cF}$ is not a metric, but a pseudometric on $\XX$.

There is also the notion of (Weyl) \emph{mean equicontinuity} (or just
\emph{mean equicontinuity}), which is equicontinuity with respect to
the pseudometric
$D_{_\mathrm{W}} (x,y) \defeq \sup_{\nts\cF}D^{}_{\nts\cF}(x,y)$,
where the supremum is taken over all F{\o}lner sequences, making
$D_{_\mathrm{W}}$ manifestly independent of any F{\o}lner sequence.
However, it turns out that, in the Abelian case, Weyl mean
equicontinuity and Besicovitch $\cF$-mean equicontinuity for any
F{\o}lner sequence $\cF$ are actually equivalent; compare
\cite[Thm.~1.3]{FGL21}. With these notions, we can recall the
following result from \cite{FGL21,FG20,FGJK23}.

\begin{theorem}\label{thm:meaneq}
  Suppose\/ $(\XX,G)$ is a minimal TDS with\/ $G$ Abelian, and let\/
  $\cF$ be any F{\o}lner sequence.  Then, the following properties are
  equivalent.
  \begin{itemize}
    \item $(\XX,G)$ is Besicovitch\/ $\cF$-mean equicontinuous. 
    \item $(\XX,G)$ is Weyl mean equicontinuous.
    \item $(\XX,G)$ has a unique ergodic measure\/ $\mu$, and\/
      $(\XX,G,\mu)$ has pure-point dynamical spectrum with continuous
      eigenfunctions.
    \item The separation numbers\/ $\mathrm{Sep}\,(\XX,\delta,\nu)$
      are finite. \qed
  \end{itemize}
\end{theorem}

Recall that the $\OSD$ of mean equicontinuous, minimal dynamical
systems does not depend on the choice of the F{\o}lner sequence
\cite[Thm.~1.2]{FGJK23}. Since we are only concerned with such systems
throughout this work, the choice of $\cF$ will be immaterial.

We also note that, for mean equicontinuous dynamical systems $(\XX,G)$,
the factor map to the MEF can be obtained by identifying points with
vanishing distance in the pseudometric $D_{_\mathrm{W}}$. More
precisely, we consider the equivalence relation given by
\begin{equation}\label{eq:equiv}
  x \sim y \quad \Longleftrightarrow \quad
  D_{_\mathrm{W}}(x,y) = 0 \ts ,
\end{equation}
called the \emph{equicontinuous structure relation}.  The equivalence
class corresponding to $x$ is denoted by $[x]$, and $D_{_\mathrm{W}}$
naturally defines a metric on the quotient space $[\XX] = \XX/{\sim}$
via
$D_{_\mathrm{W}}\bigl( [x], [y] \bigr) \defeq
D_{_\mathrm{W}}(x,y)$. It is easily checked that this is
well-defined. For every subset $S \subseteq \XX$, we also use the
notation
\[
   [S] \, = \, \bigl\{ [x] : x \in S \bigr\}.
\]
Since the pseudometric $D_{_\mathrm{W}}$ is invariant under
translations, the same holds for the corresponding equivalence
relation. This shows that the group action
\[
   [x] \, \longmapsto \, [x]+t \defeq [x+t]
\]
is well-defined on the quotient space, thus turning
$\bigl( [\XX],G \bigr)$ into another TDS.  By construction, the
(quotient) map
\[
  \pi \colon \, \XX \xrightarrow{\quad} [\XX],
      \quad x \mapsto [x]
\]
intertwines (or semi-conjugates) the action of $G$ on the two spaces.
Since we assume that $(\XX,G)$ is mean equicontinuous, the action of
$\pi$ is continuous and hence defines a factor map from one TDS onto
another.  In this situation, it is also known that $([\XX],G)$
coincides with the MEF of $(\XX,G)$; see \cite[Prop.~3.13]{FGL21}, as
well as \cite{Aus,ABKL} for background.

As already indicated, our systems of interest are spaces of primitive
inflation tilings and their translation dynamical systems. We will
therefore confine ourselves to $G=\RR^d$ (or to some discrete subgroup
thereof, such as $\ZZ^d$). In the remainder of this section, we give a
brief description of the setup and the basic properties. For further
background, we refer to \cite{TAO}.

A \emph{tiling} is a covering of $\RR^d$ with tiles, which are compact
sets each of which is the closure of its interior, and which are
supposed to overlap at most on their boundaries. We only consider the
case of regular tiles, that is, tiles with a boundary of Lebesgue
measure $0$.  Tiles may also carry an extra label, to distinguish
tiles of different type, but having congruent support (or shape). We
assume that, up to translation by elements of $G$, there are only
\emph{finitely} many tile types, called \emph{prototiles}. We also
require that tiles may be joined in finitely many ways only.  To make
this more precise, if $B_{r} (a)$ is the the closed ball of radius $r$
centred at $a$, we call the collection of all those tiles in a tiling
$x$ which intersect such a ball an \emph{$r$-patch}, written as
$x \sqcap B_{r} (a)$. We require that, up to translation, the number
of different $r$-patches is finite for any radius $r$. This property
is called \emph{finite local complexity} (FLC).

On the space of all tilings constructed from a given set of
prototiles, we have a metric, known as the \emph{tiling metric}
$d(x^{}_1 , x^{}_2 )$, which is defined as
\[
   d (x^{}_1 ,x^{}_2) \, \defeq \,
   \min \bigl\{ \widetilde{d}(x^{}_1 , x^{}_2), 2^{-1/2} \bigr\},
\]
where
\[
  \widetilde{d}(x_1,x_2) \, = \, \inf \{ \varepsilon > 0 :\,
  \exists\, t^{}_1, t^{}_2 \in B_\varepsilon(0) \text{ with } 
     (x^{}_1 - t^{}_1) \sqcap B_{1/\varepsilon}(0) =
      (x^{}_2 - t^{}_2) \sqcap B_{1/\varepsilon}(0) \} \ts .
\]
In words, to be $\varepsilon$-close, the tilings must exactly agree in
a ball of radius $1/\varepsilon$ around the origin, possibly after
some $\varepsilon$-small rigid translations \cite{Sol97}. The cap at
$2^{-1/2}$ is necessary for the triangle relation to hold
\cite{LMS02}. The topology induced by the tiling metric is called the
\emph{local topology}; compare \cite{TAO}. The orbit closure of an FLC
tiling in the local topology is compact. The group $\RR^d$ acts
continuously on it by translation.  We thereby obtain a
\emph{topological tiling dynamical system} $(\XX,\RR^d)$, or TTDS for
short. For the tilings we consider, these dynamical systems will be
strictly ergodic, so that the tiling space can be constructed as the
closure of the $\RR^d$-orbit of any of its elements.

We further assume that our tiling space is invariant under a primitive
inflation rule with a uniform scaling factor.  In this procedure, a
tile, patch of tiles, or an entire tiling, is scaled by a factor
$\lambda>1$ (which necessarily is an algebraic integer; compare
\cite[Thm.~2.4]{TAO} and references given there), and the scaled tiles
are then dissected in a unique way into tiles of the original size.
Starting with a single tile, by iterated inflation, and combined with
translations, one can then generate larger and larger patches, and
then (in the limit) tilings of the entire space. An inflation
procedure is called \emph{primitive} if it eventually produces, from
any starting tile, a patch containing tiles of all types
simultaneously. Primitive inflation rules have the property that all
tilings they can generate are \emph{locally indistinguishable} (LI),
hence members of the same tiling space, with strictly ergodic
translation dynamics.

The case just described, where a scaled tile is replaced by a patch of
tiles having the same support as the scaled tile, is called a
\emph{stone inflation} \cite{TAO}. The condition of matching supports
can actually be relaxed to some extent, as long as tiles in a tiling
are replaced by patches which still cover the whole space without
(interior) overlaps. Then, there exists a stone inflation tiling which
is \emph{mutually locally derivable} (MLD, compare \cite{TAO}), often
even with a combinatorially isomorphic inflation. MLD induces a
topological conjugacy of the dynamical systems, with a conjugating map
that is uniformly local.

If a tiling space consists of aperiodic tilings, each tiling has a
unique predecessor under inflation, a property which is known as
\emph{recognisability} or \emph{unique composition} \cite{Sol98}; this
is the situation of a \emph{local inflation deflation symmetry} (LIDS)
in \cite{TAO}. If, moreover, the spectrum of the translation action
with respect to the unique invariant measure on the tiling space $\XX$
is pure point, all eigenfunctions are known to have continuous
representatives \cite{Sol07}.  The MEF is thus non-trivial, and the
factor map to the MEF is known to be one-to-one almost everywhere
(with respect to the Haar measure on the MEF), and uniformly
bounded-to-one everywhere; compare \cite{BK13}.

In view of Theorem~\ref{thm:meaneq}, we may restrict our attention to
those tiling dynamical systems that have pure-point dynamical
spectrum. That is, for the remainder of this work, we assume that
$(\XX,G)$ is a pure-point tiling dynamical system, generated by a
primitive inflation rule $\varrho$. To avoid trivialities, we further
assume $(\XX,G)$ to be non-periodic\footnote{In the fully periodic
  case, the dynamics is equicontinuous, resulting in vanishing
  $\OSD$. More generally, by a refined argument of this type, no
  periodic direction contributes to the $\OSD$.}  In this case,
$\varrho$ defines a homeomorphism on $\XX$ that induces a map on the
MEF via
\[
  \varrho \colon \,  [\XX] \xrightarrow{\quad} [\XX] \ts ,
  \quad [x] \mapsto [\varrho(x)] \ts .
\]
This is a well-defined homeomorphism on $[\XX]$, compare \cite{ABKL},
which is consistent with set notation in the sense that
\begin{equation}\label{eq:set-vs-eq-class}
  \varrho \bigl( [x] \bigr) \, = \,
  \bigl\{ \varrho(y) : y \in [x] \bigr\}.
\end{equation}
This partly justifies the slight abuse of notation in choosing the
same symbol $\varrho$ for the induced map on the MEF. Now, our
calculations of the $\OSD$ will rely on the observation that it
coincides with the upper box counting dimension of (a subspace of) the
MEF, equipped with an appropriate metric, as we explain next.

\section{The transversal}\label{sec:trans}

As a first step, we define a transversal $\XX^{}_{0} \subset \XX$ to
the TTDS $(\XX,\RR^d)$, with the property that the set of translations
$\{t \in \RR^d\, : \, \exists\, x\in\XX^{}_{0} \text{ with } x+t \in
\XX^{}_{0}\}$ is uniformly discrete (and in general no longer forms a
group). Such a transversal, also called discrete hull, had already
been used by Kellendonk \cite{Kel95}. As we shall see, the use of a
transversal will simplify the computation of the separation numbers
$\Sep(\XX,\delta,\nu)$ significantly.

\begin{remark}\label{rem:Gdiscrete}
  In the case where $G$ is discrete, such as $\ts G=\ZZ^d$, the
  construction of a transversal can actually be omitted, and the
  reader can continue directly after the proof of
  Corollary~\ref{coro:ac0}, with $\XX^{}_{0}=\XX$. Conversely, a
  discrete group $\ZZ^d$ can also be extended to $\RR^d$ via a
  standard suspension (with constant roof function), without changing
  the $\OSD$; see \cite{FGJK23}.  A discrete group $\ts G=\ZZ^d$
  naturally occurs for (symbolic) constant-length substitutions
  ($d=1$) or for constant-shape block substitutions; compare
  \cite{FG20}. Standard suspensions of such systems are naturally
  isomorphic to inflation tilings where all tiles are congruent to
  $d$-dimensional unit cubes, which are then distinguished only by
  their type label. There are further generalisations, as one can see
  from the half-hex inflation, but we do not go into further details.
  \exend
\end{remark}

To construct the transversal, we need to define so-called
\emph{control points} for all tiles, which are special reference
points with some extra properties. The control point of a single tile
$\tau$ is denoted by $\cp (\tau)$.  Now, if two tiles $\tau$, $\tau'$
of the same type occur in the same tiling $x$, they are translates of
each other, $\tau = \tau' + t$, and we request that their control
points then satisfy $\cp(\tau) = \cp(\tau') + t$. Such a translation
$t$ is called a (tile) \emph{return vector}. The $\ZZ$-span $R$ of all
return vectors is called the module of return vectors. Finally, the
module of \emph{eventual return vectors} is defined as
$\er \defeq \bigcup_{n \in \NN_0} \lambda^{-n}R$, so that
$\lambda \er = \er$. Note that $\er = R$ if and only if the algebraic
integer $\lambda$ is a unit. By construction, control points of tiles
of the same type are separated by a vector in $\er$, but for tiles of
different types, this is not automatically the case, and must be
required as an extra property. Finally, we require that the control
point of a tile inflated by $\lambda$ coincides with the control point
of one of its constituent tiles.  The next result shows that control
points with the required properties can always be chosen.  Further
useful requirements will be discussed as we proceed.

\begin{lemma}\label{lemma:cp}
  For the tiles in a primitive inflation tiling with inflation
  factor\/ $\lambda$, it is possible to choose control points with the
  following properties.
\begin{enumerate}\itemsep=2pt
  \item For any two tiles in the tiling, the difference of their
    control points is contained in the module\/ $\er$ of eventual
    return vectors.
  \item The control point\/ $\lambda\,\cp(\tau)$ of an inflated tile\/
    $\lambda\,\tau$ coincides with the control point of one of its
    constituent tiles.
  \item Distinct tiles have distinct control points.
\end{enumerate}
\end{lemma}

\begin{proof}[Sketch of proof]
  We essentially follow Kenyon \cite{Ken94}. For each tile $\tau$ of
  type $i$, we choose a fixed tile in its inflated version, and call
  it the successor $\suc(\tau)$ of $\tau$.  We want the control point
  of $\tau$ to coincide with the control point of $\suc(\tau)$ after
  inflation, $\cp (\lambda \tau) = \cp (\suc (\tau))$. Since
  $\cp (\lambda\tau) = \lambda\ts \cp (\tau)$, we certainly have
  $\cp(\tau) \in \lambda^{-1}\supp(\suc(\tau))$. Now, also
  $\suc(\tau)$ has a successor, $\suc(\suc(\tau))=\suc^2(\tau)$, so
  that $\cp(\tau)$ must be contained also in
  $\lambda^{-2}\supp(\suc^2(\tau))$. Iterating this, we get
\[
  \cp(\tau) \in \bigcap_{n\in\NN}\lambda^{-n}\supp
  \bigl(\suc^n(\tau)\bigr) .
\]
The latter set is non-empty and consists of a single point, so that we
have fixed the control points of all tile types. By construction, they
satisfy requirement~(2).

In order to fulfil also the first requirement, we observe that, for a
primitive inflation rule with finitely many prototiles, there exists a
power $n$ of the inflation and a tile type $i$ such that the $n^{th}$
inflation of each tile type contains a tile of type $i$.  Choosing
such a tile as successor produces control points satisfying also
requirement~(1), first for the $n^{th}$ power of the inflation, but
these control points can also be used for the original inflation.

Finally, to avoid control points on a boundary, we may choose a
suitable successor in the \emph{interior} of an inflated tile, with
positive distance to the boundary of the latter. This shows that
requirement (3) can also be satisfied, without interference with the
previous steps.
\end{proof}

\begin{remark}\label{rem:cp}
  In some situations, it is desirable to require further properties of
  the control points. For instance, in one dimension, it would be
  convenient to choose the left endpoints of the tiles as control
  points, which may be in conflict with requirement (1), for instance
  outside the case of height $1$ in constant-length
  substitutions. This problem can be cured by passing to a tiling
  which is MLD to the original one, by using a suitable return word
  encoding. \exend
\end{remark}

\begin{remark}\label{rem:cp2}
  An important further requirement is that the control point set
  should be MLD with the tiling.  Indeed, we should think of control
  points as carrying an extra label, which identifies the type of the
  tile a control point represents. In this sense, the set of control
  points $\cp(x)$ of a tiling $x$ is a subset of $\RR^d \times S$,
  with $S$ the set of tile types. Such a set of labelled control
  points is MLD with the tiling it is derived from.  Often, the extra
  label is not necessary, for instance if the tile type can be
  determined from the local neighbourhood of other control
  points. Thus, to avoid overloaded notation, we usually silently
  drop the type label from a control point when it is not needed.
  \exend
\end{remark}

In the following, we always assume that suitable control points have
been chosen.

\begin{definition}\label{def:trans}
  Suppose that control points according to Lemma~\ref{lemma:cp} have
  been chosen. Denoting by $\cp(x)$ the set of control points of all
  tiles in a tiling $x$, we define the transversal
  $\XX^{}_{0}\subset\XX$ as $\XX^{}_{0} = \{x\in\XX : 0 \in \cp(x)\}$.
\end{definition}

\begin{coro}\label{coro:X0}
  By construction, the transversal\/ $\XX^{}_{0}$ has the following
  properties.
\begin{enumerate}\itemsep=2pt
\item The set\/
  $\{ t \in \RR^d : \exists\, x \in \XX^{}_{0} \text{ with } t \in
  \cp(x) \}$ is uniformly discrete, with a minimal distance\/ $\rho>0$
  between distinct points.
\item The inflation rule maps\/ $\XX^{}_{0}$ into itself.
\end{enumerate}
\end{coro}

\begin{proof}
  The set of control points $\cp(x)$ of a primitive inflation tiling
  $x\in\XX^{}_{0}$ with pure-point spectrum is necessarily a Meyer set
  \cite{LS08}. As such, its difference set, $\vL = \cp(x) - \cp(x)$,
  which by the minimality of the dynamics is independent of $x$, is
  uniformly discrete, and in fact a Meyer set itself. Hence, the
  difference $p-q$ of any two control points $p\in\cp(x)$ and
  $q\in\cp(y)$, $x,y\in\XX^{}_{0}$, is an element of $\vL - \vL$, so
  that $|p-q|$ is either $0$, or uniformly bounded away from $0$ by a
  positive constant $\rho$.
  
  By property (2) of Lemma~\ref{lemma:cp}, the inflation maps control
  points to control points, which means
  $\lambda\,\cp(x)\subset\cp(x)$.  The tilings in $\XX^{}_{0}$ are
  those with a control point at $0$, so that $\XX^{}_{0}$ is mapped
  into itself.
\end{proof}

We now define the separation number $\Sep(\XX^{}_{0},\delta,\nu)$ as
the maximal cardinality of any $(\delta,\nu)$-separated subset
$S\subseteq\XX^{}_{0}$.  Note that, since $S$ is also a subset of
$\XX$, $(\delta,\nu)$-separation is well-defined also for any subset
of $\XX^{}_{0}$. Clearly, $\Sep(\XX^{}_{0},\delta,\nu)$ is smaller
than $\Sep(\XX,\delta,\nu)$, but it is not much smaller, so that we
can express the $\OSD$ in terms of $\Sep(\XX^{}_{0},\delta,\nu)$.

But before we do so, we note that, for $0<\delta'<\delta$, we have
$D^{}_{\nts\delta'}(x,y)\geqslant D^{}_{\nts\delta}(x,y)$ for any $x,y\in\XX$,
and any $(\delta,\nu)$-separated set is also
$(\delta',\nu)$-separated.  Hence, in the definition of the $\OSD$, we
can restrict the supremum over $\delta$ to any interval $(0,\delta_0]$
with a conveniently chosen $\delta_0$, without affecting the $\OSD$.

Further, since $D^{}_{\nts\delta}$ fails to be a pseudometric on $\XX$,
we need to make use of the following property, which replaces the
triangle relation.

\begin{lemma}\label{lemma:modtrianglerel}
  For\/ $x,y,z\in\XX$ and\/ $\delta>0$, we have\/
  $ D^{}_{2\delta}(x,z) \leqslant D^{}_{\nts\delta}(x,y) +
  D^{}_{\nts\delta}(y,z)$.
\end{lemma}

\begin{proof}
  By the triangle inequality for the metric $d$, we observe that
  $d(x+t,z+t)\geqslant 2\delta$ requires at least one of the relations
  $d(x+t,y+t)\geqslant\delta$ and $d(y+t,z+t)\geqslant\delta$ to
  hold. That is, if $t\in\vD^{}_{2\delta}(x,z)$, we need
  $t\in\vD^{}_{\delta}(x,y)$ or $t\in\vD^{}_{\delta}(y,z)$. Hence,
\[
    \vD^{}_{2\delta}(x,z) \, \subseteq \, \vD^{}_{\delta}(x,y)
    \, \cup \,  \vD^{}_{\delta}(y,z) \ts .
\]
  Taking the upper density on both sides proves the claim.
\end{proof}

\begin{lemma}\label{lemma:sep}
  For any fixed\/ $\delta$ with $0<\delta<\rho/2$, where $\rho$ is the
  constant from Corollary~$\ref{coro:X0}$, there exists a constant\/
  $C_\delta$ such that, for all\/ $\nu>0$, one has\/
  $\, \Sep(\XX,2\delta,\nu) \leqslant C_\delta \,
  \Sep(\XX^{}_{0},\delta,\nu) $.
\end{lemma}

\begin{proof}
  For translations $t \in \RR^d$ with $|t|<\delta$, and any
  $x \in \XX^{}_{0}$, we have $d(x,x+t)<\delta$. Then,
  $\vD^{}_{\delta}(x,x+t)$ is the empty set, and we have
  $D^{}_{\nts\delta}(x,x+t) = 0$. If we now compare $x^{}_1 + t^{}_1$
  and $x^{}_2 + t^{}_2$ for $x^{}_i \in \XX^{}_{0}$ and
  $t^{}_i \in \RR^d$, Lemma~\ref{lemma:modtrianglerel} implies that
\[
  D^{}_{2\delta}(x^{}_1 + t^{}_1 , x^{}_2 + t^{}_2 )
  \, = \, D^{}_{2\delta}(x^{}_1 , x^{}_2 + t^{}_2 - t^{}_1 )
  \, \leqslant \,  D^{}_{\nts\delta}(x^{}_1 , x^{}_2 ) +
     D^{}_{\nts\delta}( x^{}_2 , x^{}_2 + t^{}_2 - t^{}_1 ) \ts .
\]
If $|\ts t^{}_2 - t^{}_1 | < \delta$, this reduces to
$D^{}_{2\delta}(x^{}_1 + t^{}_1 , x^{}_2 + t^{}_2 ) \leqslant
D^{}_{\nts\delta}(x^{}_1 , x^{}_2 )$. Hence, $x^{}_1 + t^{}_1$ and
$x^{}_2 + t^{}_2$ can only be $(2\delta,\nu)$-separated if either
$D^{}_{\nts\delta}(x^{}_1 , x^{}_2 ) \geqslant \nu$ or
$|t^{}_2 - t^{}_1| \geqslant \delta$.

The set $\cp(x)$ of an FLC tiling $x$ is a Delone set. There is a
radius $r$ such that the set of $r$-balls centred at control points
covers $\RR^d$, and due to the minimality of the dynamics, this radius
$r$ is the same for all $x\in\XX^{}_{0}$. Since the set
$\{x+t : x\in\XX^{}_{0}, t\in B_r (0)\}$ covers $\XX$ completely, we
see that the cardinality of any $(2\delta,\nu)$-separated set in $\XX$
is bounded by $C_\delta \, \Sep(\XX^{}_{0},\delta,\nu)$, where
$C_\delta$ is a constant of the order of $(r/\delta)^d$, independent
of $\nu$.
\end{proof}

The following result is an immediate consequence of
Lemma~\ref{lemma:sep}.

\begin{coro}\label{coro:ac0}
  $\OSD(\XX,\RR^d)$ can be computed from the transversal\/
  $\XX^{}_{0}$ as
\[
  \pushQED{\qed}
  \OSD(\XX,\RR^d) \, = \, \sup_{\delta>0} \,
  \limsup_{\nu {\scriptscriptstyle\searrow} 0}
  \frac{\log \bigl( \Sep(\XX^{}_{0},\delta,\nu)\bigr)}
  {-\log(\nu)} \ts . \qedhere \popQED
\]
\end{coro}

This shows that we can work entirely with $\XX^{}_{0}$. It is
therefore worth noting that, on $\XX^{}_{0}$ and for $\delta$ small
enough, $D^{}_{\nts\delta}$ is a pseudometric, which was not the case
for $\XX$; compare Lemma~\ref{lemma:modtrianglerel}. This is analogous
to the discussion of symbolic dynamical systems in
\cite[Lemma~6.2]{FG20}.

\begin{lemma}\label{lemma:pseudonorm}
  For\/ $\delta< \rho/2$, where\/ $\rho$ is the separation constant
  from Corollary~$\ref{coro:X0}$, $D^{}_{\nts\delta}$ is a pseudometric
  on\/ $\XX^{}_{0}$.
\end{lemma}

\begin{proof}
  Let $x,y,z\in\XX^{}_{0}$. We show that, for distances below
  $\delta$, the tiling metric $d$ behaves like an ultra-metric on the
  simultaneous orbit of $x,y,z$. More precisely, let $t\in\RR^d$ and
  assume $d(x-t,y-t)<\delta$. Then, $x$ and $y$ agree on a ball of
  radius $1/\delta$ around $t$, after an at most $\delta$-small rigid
  translation of one (or both) of the tilings. However, since $x$ and
  $y$ share a control point at the origin, all other control points
  either coincide, or are at least a distance $\rho$ apart. Since
  $\delta< \rho/2$, the tilings $x$ and $y$ must in fact agree exactly
  on the ball $B_{1/\delta}(t)$.

  By the same reasoning, $d(y-t,z-t)<\delta$ implies that $y$ and $z$
  agree on $B_{1/\delta}(t)$. Hence, the same holds for the pair
  $(x,z)$, implying that $d(x-t,z-t)<\delta$. Conversely,
  $d(x-t,z-t)\geqslant\delta$ now requires that one of
  $d(y-t,y-t)\geqslant\delta$ or $d(y-t,z-t)\geqslant\delta$ holds,
  which means that
\[ 
    \vD^{}_{\delta}(x,z) \, \subseteq \, \vD^{}_{\delta}(x,y) \cup
    \vD^{}_{\delta}(y,z) \ts .
\]
Taking the upper density on both sides gives
$D^{}_{\nts\delta}(x,z)\leqslant
D^{}_{\nts\delta}(x,y)+D^{}_{\nts\delta}(y,z)$.
\end{proof}

In the same way, two tilings $x,y\in\XX^{}_{0}$ can only be close in
the tiling metric, $d(x,y)< \rho/2$, if they exactly agree on some
ball around the origin. This means that two tilings that are close
must have some coincidences, that is, tiles which are shared by both.
We introduce the following terminology.

\begin{definition}\label{def:disc}
  For any two tilings $x,y\in \XX^{}_{0}$, we define two subsets of
  $\RR^d$ that, apart from a null set, are complements of each other:
\begin{enumerate}\itemsep=2pt
\item The \emph{coincidence set} $\coin(x,y)$ of $x$ and $y$, which is
  the union of the supports of all tiles that do occur in both $x$ and
  $y$.
\item The \emph{discrepancy set} $\disc(x,y)$ of $x$ and $y$, which is
  the union of the supports of all tiles of $x$ or $y$ that do not
  occur in both tilings.
  \end{enumerate}
\end{definition}

We have $\coin (x,y)=\overline{\RR^d\setminus\disc(x,y)}$ and
$\disc(x,y)=\overline{\RR^d\setminus\coin(x,y)}$. The two sets only
intersect on their boundaries, which are null sets in our
setting. With these concepts, we can now introduce another
pseudometric on $\XX^{}_{0}$ as follows.

\begin{definition}\label{def:D}
  For $x,y\in\XX^{}_{0}$ and a fixed F{\o}lner sequence
  $\cF = \{ F_n \}^{}_{n\in\NN}$, we set
\[
    D(x,y) \, \defeq \, \limsup_{n\rightarrow\infty}
    \frac{|\disc(x,y)\cap F_n|}{| F_n|} \ts .
\]
\end{definition}

Note that $D$ is indeed a pseudometric. It is obviously symmetric and
positive semi-definite. Further, if a tile $\tau$ is not in
$x \cap z$, it cannot be both in $x \cap y$ and in $y \cap z$. Hence,
we get $\disc(x,z) \subseteq \disc(x,y) \cup \disc(y,z)$, and the
triangle inequality follows.

\begin{remark}
  As we already mentioned in the discussion following
  Theorem~\ref{thm:meaneq}, the choice of the F{\o}lner sequence $\cF$
  is immaterial for the value of the $\OSD$. It is worth pointing out
  that even the pseudometric $D$ defined above does not depend on
  $\cF$ for the systems at hand. Indeed, it was shown in
  \cite[Thm.~1.2]{FGL21} that, for a mean equicontinuous, minimal
  system $(\XX,G)$, the product system
  $(\XX \ts {\ts\times\ts} \XX,G)$ is pointwise uniquely ergodic. Let
  $I(x,y)$ be the indicator function that takes the value $1$ if the
  pair $(x,y)$ has a discrepancy at the origin and $0$
  otherwise. Then, we can write $D(x,y)$ as a Birkhoff average of $I$
  along the sequence $\cF$. Due to unique ergodicity of the orbit
  closure of $(x,y)$, this average is independent of $\cF$. By the
  same reasoning, we see that the $\limsup$ in the definition of $D$
  can actually be replaced by a limit. \exend
\end{remark}

The pseudometric $D(x,y)$ is closely related to the pseudometric
$D^{}_{\nts\delta} (x,y)$ defined in \eqref{eq:Ddelta}.

\begin{lemma}\label{lemma:equivD}
  With respect to a fixed F{\o}lner sequence\/
  $\cF = \{ F_n \}^{}_{n\in\NN}$, the pseudometrics\/ $D(x,y)$ and\/
  $D^{}_{\nts\delta} (x,y)$, with\/ $\delta \in (0, \rho/2)$, are all
  Lipschitz equivalent.
\end{lemma}

\begin{proof}
  Recall that two tilings $x,y\in\XX^{}_{0}$ can be close in the
  tiling metric only if their coincidence set is non-trivial. As a
  consequence, the same also holds if they are close in any of the
  pseudometrics $D$ and $D^{}_{\nts\delta}$. The set
  $\vD^{}_{\delta}(x,y)$ then is the discrepancy set $\disc(x,y)$,
  thickened by an extra margin of thickness $1/\delta$,
\[
    \disc(x,y) \, \subseteq \, \vD^{}_{\delta}(x,y)
    \, \subseteq \, \disc(x,y) + B_{1/\delta}(0) \ts ,
\]
  where the $+$ denotes the Minkowski sum of sets.  Hence, we have
  $D(x,y)\leqslant D^{}_{\nts\delta}(x,y)$ for any $\delta > 0$.  On the
  other hand, since the tiles are compact, there exists a constant
  $C_\delta$ such that
  $|\supp(\tau)+B_{1/\delta}(0)| \leqslant C_\delta\, |\supp(\tau)|$
  holds for any tile $\tau$, so that we also have
  $D^{}_{\nts\delta}(x,y) \leqslant C_\delta\, D(x,y)$.  The pseudometrics
  are thus all Lipschitz equivalent.
\end{proof}

\begin{lemma}\label{lemma:equivW}
  For any F{\o}lner sequence\/ $\cF = \{ {F}_n \}_{n\in\NN}$ and any\/
  $\delta \in (0, \rho/2)$, the pseudometrics\/ $D^{}_{\nts\cF}(x,y)$
  and\/ $D_{_\mathrm{W}}(x,y)$ induce the same topology on\/
  $\XX^{}_{0}$ as\/ $D^{}_{\nts\delta}(x,y)$.
\end{lemma}

\begin{proof}
  Note first that $D_{_\mathrm{W}}$ and $D^{}_{\nts\cF}$ induce the
  same topology \cite{FGL21}; see also our discussion above.  As
  $D_{_\mathrm{W}}$ is independent of the choice of $\cF$, the induced
  topology on $\XX^{}_{0}$ does not depend on $\cF$ either.  Since
  $d(x,y) \geqslant \delta \, \one_{\vD^{}_{\delta}(x,y)}(0)$, one
  easily gets  $\delta \, D^{}_{\nts\delta}(x,y) \leqslant
  D^{}_{\nts\cF}(x,y)$. Conversely, suppose that
  $D^{}_{\nts\delta}(x_n,x) \xrightarrow{n\rightarrow\infty} 0$.  Since
  the pseudometrics $D_\varepsilon$ are all equivalent, for any
  $\varepsilon>0$ there exists some $n_\varepsilon\in\NN$ such that
  $D_\varepsilon(x_n,x) < \varepsilon$ for $n>n_\varepsilon$.  Then,
  we obtain 
\[
  D^{}_{\nts\cF}(x_n,x) \, \leqslant \,
  D_\varepsilon(x_n,x) + \varepsilon \bigl(1 - D_\varepsilon(x_n,x)\bigr)
  \, \leqslant \, 2 \ts \varepsilon
\]
for all $n>n_\varepsilon$. Since $\varepsilon > 0$ was arbitrary, this
shows that $D^{}_{\nts\cF}(x_n,x) \xrightarrow{n\rightarrow\infty} 0$.
\end{proof}

Recall now the equivalence relation \eqref{eq:equiv} induced by the
pseudometric $D_{_\mathrm{W}}$ on $\XX$ via
\[
  x \sim y \; \Longleftrightarrow \, D_{_\mathrm{W}}(x,y) = 0 \ts .
\]
When $x,y \in \XX^{}_{0}$, this is also equivalent to $D(x,y) = 0$, by
Lemma~\ref{lemma:equivW}. Hence, $D$ extends to a well-defined metric
on $[\XX^{}_{0}]$, given by $D([x],[y]) = D(x,y)$, whenever
$x,y \in \XX^{}_{0}$.  Note that since $D$ and $D_{_\mathrm{W}}$
induce the same topology, the projection to the MEF remains a
continuous map from $(\XX^{}_{0},d)$ to
$\bigl( [\XX^{}_{0}] , D \bigr)$.
 
\begin{theorem}\label{thm:osc-upper-box}
  $\OSD(\XX,\RR^d)$ coincides with the upper box counting dimension of
  the metric space\/ $\bigl( [\XX^{}_{0}] , D \bigr)$.
\end{theorem}

\begin{proof}
  Let $\delta \in (0,\rho/2)$ and recall from
  Lemma~\ref{lemma:pseudonorm} that $D^{}_{\nts\delta}$ is a
  pseudometric on $\XX^{}_{0}$. Identifying points in $\XX^{}_{0}$
  with vanishing distance in the pseudometric $D^{}_{\nts\delta}$
  yields a quotient metric space
  $(\XX^{\delta}_{0} , D^{}_{\nts\delta})$ with $\XX^{\delta}_{0}$
  consisting of equivalence classes $[x]^{}_{\delta} \in
  \XX^{}_{0}$. Since $\nu$-separated sets in this space are in
  one-to-one correspondence with $(\delta,\nu)$-separated sets in
  $\XX^{}_{0}$, we obtain that the upper box counting dimension of
  $(\XX^{\delta}_{0} , D^{}_{\nts\delta})$ can be written as
\begin{equation}\label{eq:D-delta-dimension}
  \limsup_{\nu {\scriptscriptstyle\searrow} 0}
  \frac{\log \bigl( \Sep(\XX^{}_{0},\delta,\nu) \bigr)}
  {-\log(\nu)} \ts .
\end{equation}
By Lemmas~\ref{lemma:equivD} and \ref{lemma:equivW}, the map
$ \psi: \XX^{\delta}_{0} \xrightarrow{\quad} [\XX^{}_{0}]$ with
$[x]^{}_{\delta} \mapsto [x] $ is well-defined and bi-Lipshitz
continuous. Hence, the upper box counting dimension of
$([\XX^{}_{0}],D)$ coincides with the corresponding dimension of
$\bigl( \XX^{\delta}_{0} , D^{}_{\nts\delta} \bigr)$ given in
\eqref{eq:D-delta-dimension}. Up to taking a supremum over all
$\delta \in (0,\rho/2)$, this is precisely the expression for
$\OSD(\XX,\RR^d)$ given in Corollary~\ref{coro:ac0}.
\end{proof}

Let us now turn to the action of the inflation map on the transversal,
which will provide a powerful tool for the computation of the $\OSD$.

\section{The iterated function system}\label{sec:iter}

In Corollary~\ref{coro:X0}, we have seen that the inflation $\varrho$
maps the transversal $\XX^{}_{0}$ into itself. As we shall see, this
induces a map of $[\XX^{}_{0}]$ into itself.  Our plan is to show that
an appropriate partition of $[\XX^{}_{0}]$ is the fixed point of a
graph-directed iterated function system. In \cite{FG20}, a similar,
but simpler construction was used for constant-length substitutions,
which cannot be extended to the general tiling case.

We start by showing that the inflation indeed acts as a contraction on
$[\XX^{}_{0}]$. We recall a result of Solomyak \cite{Sol97}, who
showed that the dynamical system $(\XX,\RR^d)$ has pure-point spectrum
if and only if, for all $x \in \XX$ and every return vector $\trv$,
the asymptotic density of the discrepancy set
$\disc \bigl( \varrho^{n}(x) , \varrho^{n}(x+\trv )\bigr)$ tends to
zero as $n \to \infty$, that is,
\begin{equation} \label{eq:SolCond}
  \lim_{n\to\infty} D \bigl( \varrho^{n}(x),
     \varrho^{n}(x+\trv ) \bigr) \, = \, 0 \ts .
\end{equation}

Condition \eqref{eq:SolCond} can be verified with the help of the
overlap algorithm \cite{Sol97}, or variants thereof
\cite{SS02,AL11,AGL14}.  An \emph{overlap} in a pair of tilings $x$,
$x+ \trv$ consists of a pair of tiles, one from $x$ and one from
$x+ \trv$, such that their supports have an intersection with
non-empty interior. Such an overlap is like a new kind of tile. Its
support is the intersection of the supports of the two tiles, and its
type is a combination of the two tile types involved, together with
the offset of the two tiles relative to the support of the
overlap. Effectively, positioning $x$ and $x+ \trv$ on top of each
other, this procedure dissects the resulting pattern into compact
pieces, the overlaps.

Overlaps form equivalence classes under the translation action,
thereby defining \emph{overlap types}. We can define an induced
inflation action on the overlaps, by inflating the two tiles of an
overlap, and then splitting the overlap of the two inflated tiles into
tile overlaps. Solomyak has shown that, for FLC inflation tilings with
an inflation factor that is a \emph{Pisot--Vijayaraghavan} (PV)
number, the total number of different overlap types remains bounded
when $r$ is varied and the inflation iterated. The resulting
\emph{overlap inflation} is denoted by $\varrho^{}_{\oo}$, and scales
with the same inflation factor as $\varrho$. It is no longer
primitive, however.

In fact, there is a special subset of overlaps, the \emph{coincidence
  overlaps}, which consist of pairs of tiles of the same type at the
same position.  These transform among themselves under inflation, and
form a primitive subsystem. The remaining overlaps are called
\emph{discrepancy overlaps}.  The asymptotic density of the set
covered by the discrepancy overlaps can shrink under inflation only if
new coincidences are produced from discrepancy overlaps. We have the
following result of Solomyak.

\begin{theorem}[\cite{Sol97}]\label{thm:sol}
  The TTDS\/ $(\XX,\RR^d)$ has pure-point spectrum if and only if
  every overlap eventually produces a coincidence overlap under the
  overlap inflation.  \qed
\end{theorem}

\begin{remark}\label{rem:bp}
  It may actually be advantageous to dissect the discrepancy region
  into different pieces, such as the \emph{balanced pair} \cite{SS02} 
  discrepancies in the one-dimensional case. The advantage is that
  discrepancy overlaps can then be combined into natural groups,
  which always occur together.  \exend
\end{remark}

\begin{example}\label{ex:fibo0}
  Let us illustrate discrepancies and overlaps with a simple example,
  the Fibonacci tiling \cite[Ex.~4.6]{TAO}. It is generated by an
  inflation
\[
  \varrho_{_\mathrm{F}} : \quad
  a \rightarrow ab, \quad b \rightarrow a,
\]
where $a$ and $b$ are two intervals of length
$\phi=\frac{1}{2}(1+\sqrt{5}\,)$ and $1$, respectively. The inflation
factor is $\phi$. We fist look at balanced pairs, which are
(unordered) pairs of patches of tiles having the same support. A
balanced pair is irreducible if the two patches have no inner vertices
in common. Otherwise, it can be split into a sequence of irreducible
ones at the shared inner vertices. Obviously, there are two balanced
pairs which consist of a pair of equal tiles, $A=(a,a)$ and $B=(b,b)$.
These are called \emph{coincidences}. If we superimpose two Fibonacci
tilings in such a way that they share some of the tiles, but not all,
we can decompose this bi-layer tiling into a sequence of irreducible
balanced pairs. It is not difficult to see that only a small set of
further balanced pairs occurs in this decomposition, namely
$C=(ab,ba)$ and $D=(aab,baa)$; compare \cite{SS02}.

The original inflation $\varrho$ can be extended to balanced pairs, by
inflating the two patches and then decomposing the result into
irreducible balanced pairs; compare Figure~\ref{fig:olap}. We obtain
\[
    A \rightarrow AB, \quad B \rightarrow A, \quad C \rightarrow AC, 
    \quad D \rightarrow ACC.
\]
We see that both discrepancies, $C$ and $D$, produce a coincidence $A$
under inflation, which proves that the Fibonacci TTDS has pure-point
spectrum. Moreover, $D$ is never produced by inflation, and so the
$D$s are quickly eliminated under inflation, and need not be
considered further. $C$ is therefore the only essential balanced pair
discrepancy remaining.

\begin{figure}
\begin{center}
\includegraphics[width=0.8\textwidth]{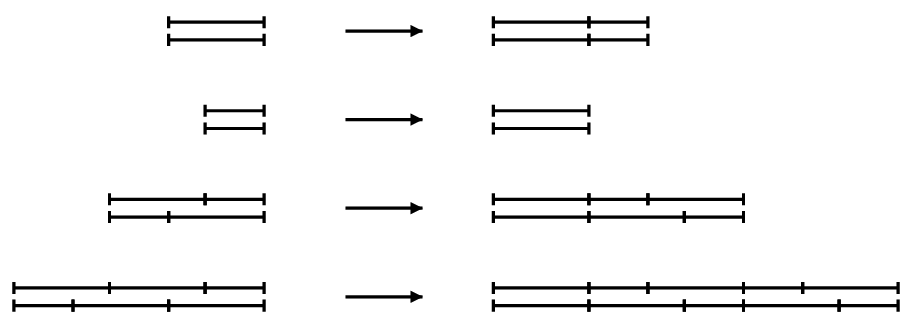}
\end{center}
\caption{\label{fig:olap} Inflation of balanced pair overlaps $A$ --
  $D$ (from top to bottom). }
\end{figure}

We will now decompose $C$ into discrepancy overlaps in the sense
defined above. The two patches in $C$ have two interior vertices that
are not shared. We cut $C$ at those vertices into three pieces which
are the three discrepancy overlaps, $C=C_1C_2C_3$, as in
Figure~\ref{fig:olap}. Here, $C_1$ consists of a pair of an $a$ tile
and a $b$ tile, which share their left vertices, resulting in an
overlap region of length $1$.  $C_3$ is analogous, except that here
the right vertices are shared.  $C_2$ consists of a pair of $a$ tiles,
which are shifted by a distance $1$ with respect to each other,
resulting in an overlap region of length $\phi-1$. The coincidences
$A$ and $B$ are also coincidence overlaps. Analogously to balanced
pairs, we can extend the inflation to overlaps, giving
\[
    C_1 \rightarrow A \ts , \quad C_2 \rightarrow C_1 ,
    \quad C_3 \rightarrow C_2C_3 \ts .
\]
We see that, after three iterations, every discrepancy overlap $C_i$
has produced at least one coincidence, which proves again the
pure-point nature of the Fibonacci tiling spectrum.  \exend
\end{example}

Since we \emph{assume} a system with topological pure-point spectrum,
Theorem~\ref{thm:sol} implies that some fixed power of the overlap
inflation produces a coincidence from every overlap.  This can now be
used to prove the following result.

\begin{prop}\label{prop:X0contract}
  Assume that the TTDS $(\XX,\RR^d)$ has topological pure-point
  spectrum. Then, there exist numbers\/ $n \in \NN$ and\/
  $\rc \in (0, 1)$, such that, for all\/
  $x,y \in \XX^{}_{0}$,
\[
    D \bigl( \varrho^n(x),\varrho^n(y)\bigr)
    \, \leqslant \, \rc \ts D(x,y) \ts .
\]
\end{prop}

\begin{proof}
  We first note that, by the construction of our control points, there
  exists a power $\varrho^{k_1}$ of the inflation such that, for any
  $x,y\in\XX^{}_{0}$, the tilings $\varrho^{k_1}(x)$ and
  $\varrho^{k_1}(y)$ have some coincidences. Starting from such a
  pair, by Theorem~\ref{thm:sol}, there exists a power $\varrho^{k_2}$
  that produces a coincidence from every overlap. Hence, there exists
  $\rc <1$ such that
\[
  D \bigl( \varrho^{k_1+k_2}(x),\varrho^{k_1+k_2}(y) \bigr)
  \, \leqslant \, \rc \ts D
  \bigl( \varrho^{k_1}(x),\varrho^{k_1}(y)
  \bigr) \, \leqslant \, \rc \ts D(x,y) \ts .
\]
Consequently, $\varrho^{k_1+k_2}$ is a contraction on $\XX^{}_{0}$
with contraction constant $\rc$ relative to the
pseudometric $D$.
\end{proof}

As all powers of $\varrho$ generate the same hull, we may assume
$D(\varrho(x),\varrho(y)) \leqslant \rc \ts D(x,y)$ from
now on.  We are now ready to construct appropriately defined
components of $[\XX^{}_{0}]$ as the unique fixed point of a
(graph-directed) iterated function system (GIFS). For this purpose, we
partition $\XX^{}_{0}$ into cylinder sets $\XX_i$. Labelling tile
types with natural numbers $i \in \{ 1,2,\ldots, n_{\tau} \}$, the
cylinder $\XX_i$ consists of all tilings in $\XX^{}_{0}$ that have a
tile of type $i$ with its control point at the origin. Due to the
recognisability property of aperiodic inflation tilings, the tile at
the origin of any tiling in $\XX_i$ is contained in a unique
supertile, which must be positioned so that it has a tile of type $i$
at the origin.  Hence, $\XX_i$ is a disjoint union of smaller cylinder
sets,
\begin{equation}\label{eq:ifs1}
  \XX_i \, = \, \bigcup_j\bigcup_k
  \bigl(\varrho(\XX_j) - t_{ij;k} \bigr) ,
\end{equation}
where $j$ runs over all supertile types which contain a tile of type
$i$, and the $t_{ij;k}$ are the positions of such a tile within the
supertile.

Applying the projection to the MEF on both sides of
Eq.~\eqref{eq:ifs1} (and using that $\varrho$ commutes with this
factor map) yields a similar relation with $\XX_i$ replaced by
$[\XX_i]$, for all $i$. In contrast to the approach of \cite{FG20},
the unions will no longer be completely disjoint, as there may be
fibres which are split over several cylinders.  These overlaps will be
small, however.  Let us now translate \eqref{eq:ifs1} (with $\XX_i$
replaced by $[\XX_i]$) into the language of a GIFS; compare
\cite{EG99}.

\begin{definition}
  We call
  $\bigl\{ \cG, \{(X_i,d_i)\}_{i \in V}, \{f_e \}_{e \in E} \bigr\}$ a
  GIFS if $\cG$ is a graph with vertex set $V$ and edge set $E$, each
  $(X_i,d_i)$ is a compact metric space, and for each $e \in E_{ij}$,
  the set of edges from $i$ to $j$, we have that $f_e: X_j \to X_i$ is
  a strict contraction.  A tuple $\{Y_i \}_{i \in V}$ of compact sets
  $Y_i \subseteq X_i$ is called an \emph{invariant list} of the GIFS if
  $ Y_i = \bigcup_{j \in V} \bigcup_{e \in E_{ij}} f_e(Y_j)$ for all
  $i \in V$.
\end{definition}

Here, we let $V$ be the set of tile types and may choose
$(X_i,d_i) = \bigl( [\XX_i], D \bigr)$ for the metric spaces. Note
that each $[\XX_i]$ is compact as the image of a compact set under a
continuous map. There is an edge $e$ for each occurrence of a tile of
type $i$ in a supertile of type $j$, indexed by $k$. If $t_{ij;k}$ is
the corresponding position in the supertile, we set
$f_e \bigl( [x] \bigr) \defeq \varrho \bigl( [x] \bigr) - t_{ij;k}$.
Since the metric $D$ is translation invariant, we conclude from
Proposition~\ref{prop:X0contract} that the maps $f_e$ are indeed all
strict contractions.  With this notation, the projection of
\eqref{eq:ifs1} to the MEF yields
\begin{equation}\label{eq:ifs}
  [\XX_i] \, = \, \bigcup_{j\in V}\bigcup_{e\in E_{ij}}
  f_e \bigl( [\XX_j] \bigr) .
\end{equation}
That is, $\bigl\{ [\XX_i] \bigr\}_{i \in V}$ is an invariant list of
this GIFS and the $\OSD$ is nothing but the (upper) box dimension of
$[\XX^{}_{0}] = \bigcup_{i \in V} [\XX_i]$ relative to $D$.  If $\cG$
is strongly connected, which is the case for primitive inflation
tilings, the invariant list is in fact unique.  We note that the
fractal dimension of the fixed point space
$\bigl\{ [\XX_i] \bigr\}_{i \in V}$, relative to $D$, can be estimated
from the graph $\cG$ and the contraction rates of the maps $f_e$.
This is a classic result if each $f_e$ is a similarity on a Euclidean
space. A version for more general contractions on metric spaces,
involving upper and lower bounds on the contraction rates, can be
found in \cite{EG99}. We recall it here for the sake of being
self-contained. In order to obtain a lower bound, we require an
appropriate separation condition.

\begin{definition}
  We say that the GIFS
  $\bigl( \cG, \{X_i \}^{}_{i \in V}, \{f_e \}^{}_{e \in E} \bigr)$
  with invariant list $\{ Y_i \}^{}_{i \in V}$ satisfies the
  \emph{strong open set condition} (SOSC) if there is a collection of
  open sets $\{U_i \}^{}_{i \in V}$ with the following properties.
\begin{enumerate}\itemsep=2pt
\item For all $i,j \in V$, $e \in E_{ij}$, we have
  $f_e(U_j)\subseteq U_i$.
\item For all $i,j,j' \in V$, $e\in E_{ij}$, $e'\in E_{ij'}$,
  $e \ne e'$, we have $f_e(U_j) \cap f_{e'}(U_{j'}) = \varnothing$.
\item For all $i \in V$, we have $U_i \cap Y_i \ne \varnothing$.
\end{enumerate}
\end{definition}

The following is a special case (uniform bounds on all contractions)
of a slightly more general result in \cite{EG99}. To formulate the
result, let $M$ denote the incidence matrix of the graph $\cG$,
with entries $M_{ij} = \card ( E_{ij} )$, and $\rho^{}_{\mathrm{sp}}(M)$ the
corresponding spectral radius.

\begin{theorem}[\cite{EG99}]\label{thm:EG-dimension-bounds}
  Let\/
  $\bigl( \cG, \{X_i \}^{}_{i \in V}, \{f_e \}^{}_{e \in E} \bigr)$ be
  a GIFS with invariant list\/ $\{ Y_i \}^{}_{i \in V}$ and let\/ $\cG$
  be strongly connected. Assume that the GIFS satisfies the SOSC and
  let\/ $\dim$ denote Hausdorff dimension, lower box counting
  dimension or upper box counting dimension. If there are numbers\/
  $0 < \rc \leqslant r^{\ts\prime}_{\nts\mathrm{c}} <1$ with
\[
  \rc \ts d^{}_j(x,y) \, \leqslant \,
  d^{}_i \bigl( f^{}_e(x) , f^{}_e (y) \bigr) \, \leqslant \,
  r^{\ts\prime}_{\nts\mathrm{c}} \ts d^{}_j(x,y) \ts ,
\]
for all\/ $e \in E_{ij}$, $i,j \in V$ and\/ $x,y \in X_j$, then, for
all\/ $i \in V$, we have that
\[
  \pushQED{\qed}
  \frac{\log \bigl(\rho^{}_{\mathrm{sp}}(M)\bigr)}
       {- \log(\rc)} \, \leqslant
  \, \dim ( Y^{}_i ) \, \leqslant \, \frac{\log
    \bigl(\rho^{}_{\mathrm{sp}}(M)\bigr)}
    {- \log(r^{\ts\prime}_{\nts\mathrm{c}})} \ts .
  \qedhere \popQED
\]
\end{theorem}

In view of this result, we naturally strive to verify the SOSC for the
GIFS at hand.  For this purpose, we define sets
$U_j = \bigl\{[x] \in [\XX_j]: [x] \subset \XX_j \bigr\}$, consisting
of fibres with are completely contained in a single cylinder.  Since
generic fibres contain a single element \cite{ABKL}, the sets $U_j$
are all non-empty.  Let us point out that an equivalence class
$[x] \in [\XX_j]$ always intersects $\XX_j$, but may also contain
points from $\XX\setminus\XX_j$. However, since we are dealing with
repetitive Meyer sets, it is known that there are only finitely many
possible positions for the closest control point of type $j$ to the
origin in $y$ if $y \sim x \in \XX_j$; compare
\cite[Cor.~4.17]{ABKL}. In a slightly different formulation, this
yields the following result.

\begin{lemma}\label{lem:fiber-separation}
  Given\/ $(\XX,\RR^d)$ as above, there is a characteristic constant\/
  $\kappa > 0$ with the following property. If\/ $y \sim x \in \XX_j$
  and\/ $y \notin \XX_j$, then\/ $d(x',y) > \kappa$ for all\/
  $x'\in\XX_j$. \qed
\end{lemma}

This dichotomy is useful for the proof of the following observation.

\begin{lemma}\label{lemma:Uopen}
  For each\/ $j$, the set\/
  $U_j = \bigl\{ [x] \in [\XX_j]: [x] \subseteq \XX_j \bigr\}$ is open
  in\/ $[\XX_j]$.
\end{lemma}

\begin{proof}
  Assume the contrary. Then, there exists a point $[x]\in U_j$ such
  that every open neighbourhood of $[x]$ contains a point in
  $[\XX_j] \setminus U_j$. That is, we can find a sequence of points
  $[x_n] \in [\XX_j] \setminus U_j$ such that $[x_n] \to [x]$ as
  $n \to \infty$. By the definition of $U_j$ and $[\XX_j]$, every
  equivalence class $[x_n]$ has a representative $x_n \notin \XX_j$
  and a representative $x^{\ts\prime}_n \in \XX_j$ with
  $x^{}_n \sim x^{\ts\prime}_n$.  By Lemma~\ref{lem:fiber-separation},
  the distance of $x_n$ to $\XX_j$ is bounded away from zero. This
  shows that the limit point $x^{\ts\prime}$ of an arbitrary converging
  subsequence $(x_{n_j})$ is not contained in $\XX_j$. However, due to
  the continuity of the factor map,
\[
  [x{\ts}'] \, = \, \bigl[ \ts \lim_{j \to \infty} x_{n_j} \ts \bigr]
  \, = \, \lim_{j \to \infty} [x_{n_j}] \, = \, [x] \ts .
\] 
Consequently, the fibre $[x] \in U_j$ contains an element $x{\ts}'$ in
the complement of $\XX_j$, in contradiction to the definition of
$U_j$.
\end{proof}

\begin{prop}\label{prop:sosc}
  Assume again that we have a non-periodic tiling dynamical system
  with pure-point spectrum, which is generated by a primitive
  inflation rule.  Then, the corresponding GIFS satisfies the SOSC,
  with open sets\/ $U_j$ as defined above.
\end{prop}

\begin{proof}
  Since $U_j \subseteq [\XX_j]$ is non-empty, the third condition is
  immediate.  Let $e \in E_{ij}$ and $[x] \in U_j$. Then, all points
  in the equivalence class $f_e \bigl( [x] \bigr)$ are contained in
  $f_e(\XX_j)$.  Indeed, using \eqref{eq:set-vs-eq-class} in the first
  step, we obtain
\[
  f_e \bigl( [x] \bigr) \, = \, \bigl\{ f_e(x') : x'\in [x] \bigr\}
  \, \subseteq \, f_e (\XX_j) \ts ,
\]
because $[x] \in U_j$, hence $[x]  \subseteq \XX_j$.  Since
$f_e(\XX_j) \subseteq \XX_i$, it follows that each element of the
fibre $f_e \bigl( [x] \bigr)$ is also contained in $\XX_i$, which
shows that $f_e \bigl( [x] \bigr) \in U_i$. That is,
$f_e(U_j) \subseteq U_i$, as required for the first condition.
Finally, if $e \in E_{ij'}$ and $[x{\ts}'] \in U_{j'}$ we see by
identical arguments as before that the fibre
$f_{e'} \bigl( [x'] \bigr)$ is contained in $f_{e'}(\XX_{j'})$. By
recognisability of the inflation, this set is disjoint from
$f_e(\XX_j)$. Hence, the fibres $f_{e} \bigl( [x] \bigr)$ and
$f_{e'} \bigl( [x{\ts}'] \bigr)$ have no point in common, that is,
$f_{e} \bigl( [x] \bigr) \neq f_{e'} \bigl( [x{\ts}'] \bigr)$ for all
$[x] \in U_j$ and $[x{\ts}'] \in U_{j'}$. This is precisely what is
required for the second condition.
\end{proof}

At this point, we have established that
Theorem~\ref{thm:EG-dimension-bounds} can indeed be applied to the
invariant list $\bigl\{ [\XX_i] \bigr\}_{i \in V}$. This gives bounds
for the upper box counting dimension of
$\bigl( [\XX^{}_{0}], D \bigr)$, where the latter coincides with the
$\OSD$ by Theorem~\ref{thm:osc-upper-box}. The precise values of the
bounds depend on estimates for the contraction rates of the maps
$f_e$, which are the same as the corresponding estimates for
$\varrho$. As passing to higher powers of $\varrho$ might yield better
bounds, this motivates the following definition.

\begin{definition}
  For each $n \in \NN$, let $r_n,R_n \geqslant 0$ be such that
\[
  r_n D(x,y) \, \leqslant \, D \bigl( \varrho^n(x),
  \varrho^n(y) \bigr) \, \leqslant \, R_n D(x,y) \ts ,
\]
for all $x,y \in \XX_i$ and $i \in V$, and assume that $r_n$ is
maximal ($R_n$ is minimal) with this property.
\end{definition}

We are interested in the decay properties of these contraction bounds
as $n \to \infty$.

\begin{lemma}
  The sequence\/ $(R_n)^{}_{n \in \NN}$ is sub-multiplicative, while the
  sequence\/ $(r_n)^{}_{n \in \NN}$ is super-multiplicative. In particular,
  the Lyapunov exponents
\[
  \lambda_{\eL}^{\max} \, = \lim_{n \to \infty}
  \frac{\log (R_n)}{n} \, = \inf_{n \in \NN }
  \frac{\log (R_n)}{n} \quad \text{and} \quad
  \lambda_{\eL}^{\min} \, = \lim_{n \to \infty}
  \frac{\log (r_n)}{n} \, = \, \sup_{n \in \NN }
  \frac{\log (r_n)}{n} 
\]
are well defined.
\end{lemma}

\begin{proof}
  For $n,m \in \NN$, the estimate
  $D\bigl(\varrho^{n+m}(x),\varrho^{n+m}(y)\bigr) \leqslant R_n R_m
  D(x,y)$ follows from the definition. Since the same bound holds with
  $R_n R_m$ replaced by $R_{n+m}$, and the latter is minimal with this
  property, we conclude that $R_{n+m} \leqslant R_n R_m$. The estimate
  $r_{n+m} \geqslant r_n r_m$ follows analogously. The remaining
  statements are consequences of Fekete's lemma.
\end{proof}

Let us remark that, in the present setup, the incidence matrix of the
GIFS coincides with the \emph{inflation matrix} $M$ of $\varrho$,
where $M_{ij}$ is the number of occurrences of a tile of type $i$ in
a supertile of type $j$.  Since the inflation factor $\lambda$ means
that volumes in $\RR^d$ scale with $\lambda^d$ under the action of
$\varrho$, we get $\lambda^d$ as the leading eigenvalue of $M$.

\begin{theorem}\label{thm:ac}
  For a self-similar and non-periodic TTDS\/ $(\XX,\RR^d)$ with
  pure-point spectrum, generated by a primitive inflation\/ $\varrho$
  with inflation factor\/ $\lambda$, the\/ $\OSD$ satisfies the bounds
\[
  \frac{d\log(\lambda)}{-\lambda_\eL^{\mathrm{min}}}
  \, \leqslant \,  \OSD(\XX,\RR^d) \, \leqslant \,
  \frac{d\log(\lambda)}{-\lambda_\eL^{\mathrm{max}}} \ts ,
\]
where\/ $\lambda_\eL^{\mathrm{min}}$ and\/
$\lambda_\eL^{\mathrm{max}}$ are the minimal and maximal Lyapunov
exponents, respectively.
\end{theorem}

\begin{proof}
  By Theorem~\ref{thm:osc-upper-box}, the $\OSD$ is given by the upper
  box counting dimension of the metric space
  $\bigl( [\XX^{}_{0}] , D \bigr)$, and since
  $[\XX^{}_{0}] = \bigcup_{i \in V} [\XX_i]$, this coincides with the
  maximal dimension of $[\XX_i]$, with $i \in V$. The list
  $\bigl\{[\XX_i]\bigr\}_{i \in V}$ is invariant under a (strongly
  connected) GIFS that satisfies the SOSC, due to
  Proposition~\ref{prop:sosc}. Hence, applying
  Theorem~\ref{thm:EG-dimension-bounds} to the GIFS induced by the
  inflation $\varrho^n$, we obtain
\begin{equation}\label{eq:dimfixpoint}
  \frac{n \log \bigl(\rho^{}_{\mathrm{sp}}(M)\bigr)}{-\log(r_n)}
  \, \leqslant \: \overline{\dim}^{}_B \bigl( [\XX_i] \bigr)
  \, \leqslant \, \frac{n \log
    \bigl(\rho^{}_{\mathrm{sp}}(M)\bigr)}{-\log(R_n)} \ts ,
\end{equation}
for each $n \in \NN$ and $i \in V$, where $\overline{\dim}^{}_B$
denotes the upper box counting dimension. Since $M$ is primitive with
leading eigenvalue $\lambda^d$, we have
$\rho^{}_{\mathrm{sp}}(M)=\lambda^d$. Performing the limit $n \to \infty$
on both bounds in \eqref{eq:dimfixpoint} yields the assertion.
\end{proof}

As discussed at the beginning of this section, the contraction rates
needed for these bounds are governed by the overlap inflation
$\varrho^{}_{\oo}$, or more precisely, by the part of
$\varrho^{}_{\oo}$ which acts on the discrepancy overlaps.  Let us
call this part the \emph{discrepancy inflation} $\varrho^{}_{\dc}$ and
denote its inflation matrix by $M_{\dc}$.  The substitution
$\varrho^{}_{\dc}$ produces discrepancies from discrepancies, omitting
any coincidences that would be produced by $\varrho^{}_{\oo}$. Due to
this omission, the leading eigenvalue
$\lambda^{}_{\dc} = \rho^{}_{\mathrm{sp}}(M_{\dc})$ must be smaller
than the leading eigenvalue $\lambda^d$ of $\varrho^{}_{\oo}$.

Let us now consider, for each discrepancy type $i$, the smallest set
$S_i$ of overlap types that is invariant under $\varrho^{}_{\dc}$ and
includes $i$. Let $\lambda_i$ be the leading eigenvalue of
$\varrho^{}_{\dc}$, confined to the invariant subset $S_i$ of
discrepancy types.  Given an overlap of type $i$, its geometric image
grows with the inflation factor $\lambda^d$ under the substitution
$\varrho$, whereas the total amount covered by overlaps scales
(eventually) with $\lambda_i$.  This yields the following result.

\begin{prop}\label{prop:lamdamin}
  Under the general assumptions of Theorem~$\ref{thm:ac}$, we have\/
\[
  \lambda_{\eL}^{\max} \, = \,
  \log(\lambda_{\dc}) - d\ts \log (\lambda)
  \quad \text{and} \quad
  \lambda_{\eL}^{\min} \, \geqslant \, \min_i \log(\lambda_i)
    - d\ts \log(\lambda) \ts .
\]
\end{prop}

\begin{proof}
  For $j \in V$ and $x,y \in \XX_j$ with $x \not \sim y$, let
  $v = v(x,y)$ be a vector containing the relative frequencies of
  discrepancy types (normalised to $\| v \|^{}_1 = 1$). Then, the
  ratio $D\bigl( \varrho^n(x),\varrho^n(y)\bigr)/D(x,y)$ is given by
  $ \lambda^{-dn} \| M_{\dc}^{n} v \|^{}_1, $ up to a multiplicative
  constant that is independent of $n \in \NN$ and the choices of $x,y$
  (accounting for the varying volumes of different discrepancy
  types). Due to the non-negativity of both $v$ and $M_\dc$, we
  observe that $\| M_{\dc}^{n} v \|^{}_1$ can be written as a convex
  combination of the values $\| M_{\dc}^{n} e^{}_i \|^{}_1$, where $e_i$
  denotes the Euclidean unit vector corresponding to type $i$. In
  particular,
\[
  \min_i \lambda^{-dn} \| M^n_{\dc} e^{}_i \|^{}_1
  \, \leqslant \, \lambda^{-dn} \| M^n_{\dc} v \|^{}_1
  \, \leqslant \, \max_i \lambda^{-dn} \| M^n_{\dc} e^{}_i \|^{}_1 \ts .
\]
By definition, these bounds (up to a constant) also hold for $r_n$ and
$R_n$.  Since $\lambda_i$ coincides with the growth rate of
$\| M_{\dc}^n e^{}_i \|^{}_1$, this implies
\[
  \min_i \log(\lambda_i) - d\ts \log(\lambda) \, \leqslant \,
  \lambda_{\eL}^{\min} \, \leqslant \, \lambda_{\eL}^{\max}
  \, \leqslant \, \max_i \log(\lambda_i) - d\ts \log(\lambda) \ts .
\]
Given a discrepancy of type $i$, let $x \in \XX^{}_{0}$ and $t$ be
such that $i$ appears in the overlap decomposition of $x$ and
$x+t$. This implies that there is a pattern $P$ in $x$ such that the
intersection of $P$ and $P+t$ contains an overlap of type $i$. Since
$x$ is repetitive, this pattern appears with positive frequency and we
conclude that type $i$ actually occurs with positive frequency in the
pair $( x,x+t )$. Choosing $i$ such that $\lambda_i = \lambda_{\dc}$
is maximal, we obtain that there is in fact a pair $(x,x+t)$ such that
$D\bigl(\varrho^n(x),\varrho^n(x+t)\bigr)$ decays with the slowest
possible contraction rate. That is, we have
$\lambda_{\eL}^{\max} = \log(\lambda_{\dc}) - d\ts \log(\lambda)$.
\end{proof}

\begin{coro}\label{coro:prim}
  If\/ $\varrho^{}_{\dc}$ is primitive, or if every invariant subsystem\/
  $S_i$ contains a leading eigenvalue\/ $\lambda_{\dc}$, then
\[
    \lambda_{\eL}^{\min} \, = \, \lambda_{\eL}^{\max} \, = \,
    \log(\lambda_{\dc}) - d\ts \log(\lambda) \ts ,
\]
  and hence
\[
\pushQED{\qed}    
    \OSD(\XX,\RR^d) \, = \, \frac{d\ts \log(\lambda)}
    {d \ts \log(\lambda) - \log(\lambda_{\dc})} \ts . 
\qedhere \popQED
\] 
\end{coro}

\begin{remark}\label{rem:twoinf}
  We should emphasise here again that there are two inflations
  involved, the inflation $\varrho$ acting on $[\XX^{}_{0}]$, and the
  discrepancy inflation $\varrho^{}_{\dc}$, which determines the
  contraction rates of $\varrho$.  To obtain these contraction rates,
  we only needed the combinatorial properties of $\varrho^{}_{\dc}$,
  but not a geometric realisation.  If the overlap algorithm is used
  to determine pure-pointedness, these combinatorial properties can be
  obtained as a byproduct.  \exend
\end{remark}

\begin{remark}\label{rem:bp2}
  There are examples where the balanced pair discrepancy inflation is
  primitive, whereas the overlap discrepancy inflation is not.  Using
  the latter, one would have to show that the fast-shrinking
  discrepancies in an invariant subsystem $S_i$ always occur together
  with slower-shrinking other discrepancies.  \exend
\end{remark}

\section{Product tilings}\label{sec:prod}

The simplest way to obtain higher-dimensional tilings is to build
direct products of lower-dimensional ones. Let us assume we have two
dynamical systems $(\XX,\mathbb{R}^{d_\XX})$ and
$(\YY,\mathbb{R}^{d_\YY})$, each with well-defined $\OSD$.  We are
then interested in the $\OSD$ of the product system
$(\XX \times \YY, \mathbb{R}^{d_\XX+d_\YY})$.  We equip the product
space $\XX \times \YY$ with the sup metric of the individual
metrics. In the case of tiling spaces, this is not the same as the
natural tiling metric on $\XX \times \YY$ as defined in
Section~\ref{sec:prelim}, but is equivalent to it and thus does not
affect the $\OSD$. Also, we naturally compute the $\OSD$ with respect
to the product of the two F{\o}lner sequences that are used for the
individual dynamical systems.  Product systems had already been
considered in \cite{FGJ16,FGJK23}, though with equal groups acting on
the factors and with the diagonal action on the product. As we shall
see, this difference is not essential for our purpose.

Let us first observe that the $\OSD$ can equivalently be defined with
\emph{spanning} instead of separation numbers, as already noted in
\cite{FGJK23}.  We specify the $(\delta,\nu)$-\emph{span} of a set
$S\subset\XX$ as
\[
  \{x\in\XX: \exists y \in S \text{ with }
  D^{}_{\nts\delta}(x,y)\leqslant \nu \} \ts ,
\]
and we say that $S$ is $(\delta,\nu)$-spanning if its
$(\delta,\nu)$-span is the entire space $\XX$. We then set the
\emph{spanning number} $\textrm{Span}(\XX,\delta,\nu)$ to be the
minimal cardinality of a $(\delta,\nu)$-spanning subset
$S\subseteq\XX$.  The following result was shown for $\ZZ$-actions in
\cite{FGJ16}, but the proof carries over verbatim to the case of more
general group actions; compare also the discussion in \cite{FGJK23}.

\begin{lemma}\label{lemma:sepspan}
  The separation numbers and spanning numbers satisfy the relations
\[
\pushQED{\qed}    
     \Span(\XX,\delta,\nu) \, \leqslant \, 
     \Sep(\XX,\delta,\nu) \, \leqslant \,
     \Span(\XX, \delta/2,\nu/2) .
\qedhere \popQED
\] 
\end{lemma}

This lemma shows that the $\OSD$ is unchanged by replacing the
separation numbers in its definition by the spanning numbers. Given
two product tilings $x^{}_1 \nts\times y^{}_1$ and
$x^{}_2 \times y^{}_2$, with $x^{}_i \in\XX$ and $y^{}_i \in\YY$, the
definition of the sup metric implies that
\[
  \vD^{}_{\delta} (x^{}_1 \times y^{}_1, x^{}_2 \times y^{}_2) \, = \,
  \bigl( \vD^{}_{\delta} (x^{}_1, x^{}_2) \times \RR^{d_\YY}\bigr)
  \cup \bigl( \RR^{d_\XX}\nts \times\nts \vD^{}_{\delta}
  (y^{}_1 , y^{}_2) \bigr).
\]
Taking upper densities on both sides gives
\begin{equation}\label{eq:prodmetric}
  \max \bigl\{ D^{}_{\nts\delta} (x^{}_1 , x^{}_2), 
       D^{}_{\nts\delta} (y^{}_1 , y^{}_2) \bigr\}  \, \leqslant \,  
       D^{}_{\nts\delta} ( x^{}_1 \nts {\times\ts} y^{}_1,
             x^{}_2 {\ts\times\ts} y^{}_2 )
      \, \leqslant \,  D^{}_{\nts\delta} (x^{}_1 , x^{}_2) + 
         D^{}_{\nts\delta} (y^{}_1 , y^{}_2 ) \ts  . 
\end{equation}
This leads to the following analogue of the product rules derived in
\cite[Prop.~1.3]{FGJ16}.

\begin{theorem}\label{thm:product-system}
  Let\/ $(\XX {\ts\times} \YY, \RR^{d_\XX+d_\YY})$ be a direct product
  dynamical system.  Assume that both factors have a well-defined,
  finite\/ $\OSD$ and that the limit in the definition \eqref{eq:ac}
  of the\/ $\OSD$ exists. Then, the\/ $\OSD$ of the product system is
\[
    \OSD( \XX \ts{\ts\times}\ts \YY, \RR^{d_\XX} {\times\ts} \RR^{d_\YY} ) 
    \, = \, \OSD(\XX,\RR^{d_\XX}) + \OSD( \YY, \RR^{d_\YY}) \ts .
\]
\end{theorem}

\begin{proof}
  We will construct suitable separated sets and spanning sets for
  $\XX {\ts\times} \YY$. Suppose $S^{}_1$ and $S^{}_2$ are
  $(\delta,\nu)$-separated subsets of $\XX$ and $\YY$, respectively.
  By \eqref{eq:prodmetric}, $S^{}_1 {\times\ts} S^{}_2$ is then
  $(\delta,\nu)$-separated in $\XX {\ts\times\ts} \YY$, and hence
\[
     \Sep(\XX{\ts\times}\ts\YY,\delta,\nu) \, \geqslant \,
     \Sep(\XX,\delta,\nu) \nts\cdot \Sep(\YY,\delta,\nu) \ts .
\]
Conversely, let $S^{\ts\prime}_1$ and $S^{\ts\prime}_2$ be
$(\delta,\nu/2)$-spanning subsets of $\XX$ and $\YY$, respectively.
By Eq.~\eqref{eq:prodmetric},
$S^{\ts\prime}_1 {\ts\times\ts} S^{\ts\prime}_2$ then is
$(\delta,\nu)$-spanning in $\XX {\ts\times} \YY$, and hence
\[
  \Span(\XX{\ts\times}\YY,\delta,\nu) \, \leqslant \,
  \Span(\XX,\delta,\nu/2) \nts\cdot \Span(\YY,\delta,\nu/2) \ts .
\]
These two inequalities, together with Lemma~\ref{lemma:sepspan}, prove
the claim.
\end{proof}  

\begin{remark}
  In the context of tiling dynamical systems generated by a primitive
  inflation rule, the requirement about the existence of the limit in
  Eq.~\eqref{eq:ac} can be replaced by the corresponding condition on
  the transversal. This can be translated to the requirement that the
  upper and lower box counting dimensions of
  $\bigl( [\XX^{}_{0},D] \bigr)$ coincide. The upper and lower bounds
  established via the GIFS approach in the last section hold for both
  kinds of dimensions. Therefore, as soon as these bounds agree, the
  condition is satisfied. This applies in particular to the situation
  in Corollary~\ref{coro:prim}, where a closed form expression for the
  $\OSD$ is available.
  \exend
\end{remark}

\section{Window boundaries}\label{sec:win}

It is well known that primitive inflation tilings with pure-point
spectrum are also cut-and-project tilings, in the sense that the
control point sets are regular model sets \cite{TAO}. In the
case of inflation tilings, however, we have to distinguish between
control points belonging to different tile types, each of which
will have its own window, $W_i$.

It had been observed \cite{FFIW06,Sin07} that the \emph{boundaries} of
the windows $W_i$ for the different tile types form the fixed point of
a \emph{boundary inflation} $\varrho^{}_{\bd}$, a GIFS that is
combinatorially isomorphic to the (transposed) discrepancy inflation
$\varrho^{}_{\dc}$. If we denote, in slight abuse of notation, by
$\overline{\dim}_B(\partial W)$ the maximum of the upper box
dimensions $\overline{\dim}_B(\partial W_i)$, we thus obtain, in
analogy to Eq.~\eqref{eq:dimfixpoint}, an upper bound
\begin{equation}\label{eq:dimfixpointbd}
    \overline{\dim}_B(\partial W) \, \leqslant \, 
    \frac{\log \bigl(\rho^{}_{\mathrm{sp}}(M_{\bd})\bigr)}
    {-\log(\rc)} \ts ,
\end{equation}
where $\rho^{}_{\mathrm{sp}}(M_{\bd})=\lambda_{\dc}$ is the leading
eigenvalue of the inflation matrix $M_{\bd}$ for the boundary
inflation $\varrho^{}_{\bd}$ and $\rc$ is the slowest
contraction rate of the inflation in internal space.

Conversely, \cite[Thm.~1.4]{FGJK23} gives an upper bound on the $\OSD$
in terms of $\overline{\dim}_B(\partial W)$,
\begin{equation}\label{eq:bddw}
  \OSD(\XX,\RR^d) \, \leqslant \, \frac{d^{}_{\mathrm{int}}}
  {d^{}_{\mathrm{int}} \nts - \ts \overline{\dim}_B(\partial W)} \ts ,
\end{equation}
where $d^{}_{\mathrm{int}}$ is the upper box dimension of internal
space.  Assuming that the $\OSD$ exists, this can be converted into a
lower bound
\begin{equation}\label{eq:lbddw}
  \overline{\dim}_B(\partial W) \, \geqslant \, d^{}_{\mathrm{int}}
  \cdot \frac{ \OSD(\XX,\RR^d) - 1 }{ \OSD(\XX,\RR^d)} \ts .
\end{equation}

In what follows, we only consider the case of a Euclidean internal
space, which occurs if the inflation factor $\lambda$ is a unit.
Assume further that the contraction in internal space is isotropic,
which is the case when $\lambda$ is quadratic (so $\rc=\lambda^{-1}$)
or when $\lambda$ is ternary with a complex conjugate pair of Galois
conjugates (hence $\rc=\lambda^{-\frac12}$).  Then, the upper bound
simplifies to
\begin{equation}\label{eq:dimfixpointbd2}
  \overline{\dim}_B(\partial W) \, \leqslant \,
  \frac{d^{}_{\mathrm{int}}}{d} \frac{\log (\lambda_{\dc} )}
      {\log (\lambda )} \ts ,
\end{equation}
where the dimension $d^{}_{\mathrm{int}}$ of internal space is either
$d$ or $2d$. With this and the lower bound \eqref{eq:lbddw}, one can
show the following result.

\begin{prop}\label{prop:winbd}
  Assume that a primitive inflation tiling with pure-point spectrum
  has a Euclidean internal space with isotropic contraction under the
  inflation, and that the discrepancy inflation\/ $\varrho^{}_{\dc}$ is
  primitive (or that every invariant subsystem has the same leading
  eigenvalue).  Then, the Hausdorff dimension of the window boundary
  satisfies
\[
  \overline{\dim}(\partial W) \, = \,
  \frac{d_{\mathrm{int}}}{d} \frac{\log(\lambda_{\dc})}
    {\log(\lambda)} \quad \textrm{and} \quad
    \OSD(\XX,\RR^d) \, = \, \frac{d_{\mathrm{int}}}
    {d_{\mathrm{int}}\nts - \ts \overline{\dim}_B(\partial W)} \ts ,
\]
  where\/ $d_{\mathrm{int}}$ is the dimension of internal space.
\end{prop}

\begin{proof}
  The upper bound on the window boundary dimension is given by
  \eqref{eq:dimfixpointbd2}. The corresponding lower bound follows by
  combining \eqref{eq:lbddw} with Corollary~\ref{coro:prim}, which
  yields
\[
  \overline{\dim}_B(\partial W) \, \geqslant \,
  \frac{d^{}_{\mathrm{int}} \bigl( \OSD(\XX,\RR^d) -  1 \bigr)}
  {\OSD(\XX,\RR^d)} \, = \, \frac{d^{}_{\mathrm{int}}}{d}
  \frac{\log(\lambda_\dc)}{\log(\lambda)} \ts .
\]
Again using the expression for the $\OSD$ from
Corollary~\ref{coro:prim}, a straight-forward calculation yields the
claimed relation between the $\OSD$ and the window boundary dimension.
\end{proof}

\begin{remark}
  Instead of using the lower bound \eqref{eq:lbddw}, we can also use
  the lower bound of an analogue of Theorem~\ref{thm:ac}, provided we
  can replace the SOSC by some other sufficient separation condition
  satisfied by the boundary inflation.  For the case at hand, the
  \emph{weak separation property} (WSP) (see \cite{DE05}) can be shown
  to hold, and is particularly simple to check.  In the context of a
  GIFS, its sufficiency was proved by Das and Edgar \cite{DE05}. Note,
  however, that also the WSP crucially depends on the contraction of
  internal space being isotropic. \exend
\end{remark}

\section{Examples}\label{sec:ex}

The main point of our above derivation is that, for the class of
primitive inflation tilings with pure-point spectrum, one obtains a
useful topological invariant that is also practically computable, as
we now demonstrate with several examples.

\subsection{One-dimensional tilings}

In \cite{FFIW06}, the window boundary dimensions had been determined
for a considerable number of irreducible, unimodular Pisot (or PV)
inflation tilings.  For this purpose, the leading eigenvalue of the
reduced boundary inflation matrix $M_{\bd}$ was determined. As this
eigenvalue coincides with our $\lambda_{\dc}$, these results can be
used to compute also the $\OSD$ of these tilings. We have confirmed
the results of \cite{FFIW06}.  In particular, for each of these
examples, the balanced pair discrepancy inflation turns out to be
primitive, so that there is only one Lyapunov exponent, and the
formulae of Corollary~\ref{coro:prim} are immediately applicable. We
review here some of those examples, together with further ones.

\begin{example}\label{ex:fibo}
  We start by picking up again the Fibonacci tiling of
  Example~\ref{ex:fibo0}. We have seen that, under inflation, each
  discrepancy $C$ produces one copy of itself (plus a coincidence), so
  that the inflation matrix of the discrepancy inflation is a single
  number, $\lambda_{\dc}=1$.  Alternatively, we could consider the
  inflation matrix of the discrepancy overlap inflation, which reads
  \[ 
      M_{\dc} \, = \, \left( \begin{matrix}
        0 & 1 & 0 \\ 0 & 0 & 1 \\ 0 & 0 & 1
      \end{matrix} \right)
  \]
  and has leading eigenvalue $\lambda_{\dc}=1$. The $\OSD$ thus has
  the minimal possible value, $1$, and the window boundary has
  dimension~$0$, as it must be for the boundary of an interval.
  Consistently, this is the same value as for the corresponding
  symbolic model, which (up to scale) defines a TDS that is
  topologically conjugate. Further, it is known that every Sturmian
  subshift has value $1$ for its $\OSD$; compare
  \cite[Prop.~1.4]{FGJ16}.  \exend
\end{example}

\begin{example}\label{ex:fibotwist}
  The inflation rule from \cite[Ex.~4.1]{FFIW06}, which was also
  studied in \cite{GLJJ}, is related to the square of the Fibonacci
  inflation, but with a reshuffled order of the tiles within the
  supertiles,
\[
   \widetilde{\varrho} : \quad    
  a \rightarrow aab, \quad b \rightarrow ba,
\]
with inflation factor $\lambda=\phi^2$ and the same tile lengths as
for the Fibonacci inflation. The leading discrepancy eigenvalue was
determined as $\lambda_{\dc}=1+\sqrt{2}$ in \cite{FFIW06}, so that we
get
\[
\begin{split}
  \OSD(\XX,\RR) \, & = \, \frac{2\log(\phi)}
    {2\log(\phi) - \log(1+\sqrt{2}\,)}
    \, \approx \, 11.874{\ts\ts}434 
    \quad \text{and} \\[2mm]
    \overline{\dim}_B(\partial W) \, &  = \,
    \frac{\log(1+\sqrt{2}\,)}{2\log(\phi)}
    \, \approx \, 0.915{\ts\ts}785 \ts .
\end{split}
\]
This is already quite a high complexity, with a window boundary
dimension close to the upper limit of $1$.  \exend
\end{example}

For the above two examples, where the internal space is
one-dimensional, the contraction in internal space is always
isotropic, with a contraction rate $1/\lambda$, due to $\lambda$ being
a unit.  In the ternary case, the internal space contraction is only
isotropic if the two Galois conjugates of $\lambda$ form a complex
conjugate pair. If $\lambda$ is a unit, the contraction rate then is
$1/\sqrt{\lambda}$.  The following three inflation rules are of this
kind.

\begin{example}\label{ex:tribonacci}
  One of the simpler ternary unimodular Pisot inflations is the Rauzy
  or Tribonacci inflation from \cite[Ex.~4.2]{FFIW06}, see also
  \cite[Secs.~7.4 and 7.5]{PF} or \cite{BG20},
\[ 
    a \rightarrow ab, \quad b \rightarrow ac, \quad c \rightarrow a. 
\]
Here, $\lambda \approx 1.839{\ts\ts}287$ is the largest root of
$x^3-x^2-x-1$, whereas $\lambda_{\dc} \approx 1.395{\ts\ts}337$ is the
largest root of $x^4-2x-1$. Entering this into our formulae yields
\[
      \OSD(\XX,\RR) \, = \, \frac{\log(\lambda)}{\log(\lambda) -
      \log(\lambda_{\dc})} \, \approx \, 2.205{\ts\ts}957 
      \quad \text{and} \quad \overline{\dim}_B(\partial W) \, = \,
      \frac{2\log(\lambda_{\dc})}{\log(\lambda)} 
      \, \approx \, 1.093{\ts\ts}364 \ts .
\]
  Here, the window boundary dimension is small, and the complexity low.
 
  The order of the tiles within the supertiles can also be reshuffled
  for the Tribonacci inflation \cite[Ex.~4.3]{FFIW06},
\[ 
      a \rightarrow ab, \quad b \rightarrow ca, \quad c \rightarrow a. 
\]
  Here, we still have $\lambda \approx 1.839{\ts\ts}287$ as
  the largest root of $x^3-x^2-x-1$, but now,
  $\lambda_{\dc} \approx 1.726{\ts\ts}29$ is the largest root of
  $x^6-x^5-x^4-x^2+x-1$. With these values, we obtain
\[
    \OSD(\XX,\RR) \, \approx \, 9.611{\ts\ts}125  \quad \text{and} 
    \quad \overline{\dim}_B(\partial W) \, \approx \, 1.791{\ts\ts}90 \ts .
\]
  The $\OSD$ and the window boundary dimension are considerably 
  higher, which matches nicely with the significantly more complicated
  structure of the windows \cite{PF,BG20}.     \exend
\end{example}

\begin{example}\label{ex:plastic}
  While the reshuffled Tribonacci tiling already has a rather high
  complexity, there are simple inflation tilings with even higher
  complexity. Such an example is the inflation
\[
    a \rightarrow bc, \quad b \rightarrow a, \quad c \rightarrow b \ts ,
\]
whose inflation factor is the smallest PV number, also known as the
\textit{plastic number}; compare \cite{TAO}. It is the largest root of
the polynomial $x^3-x-1$, with numerical value
$\lambda \approx 1.324{\ts\ts}72$. The corresponding discrepancy
inflation factor $\lambda_{\dc} \approx 1.314{\ts\ts}78$ is the
largest root of the polynomial $x^{13}-x^{12}-x^{10}+x^9-2x^4+x^3-1$.
With these values, we get
\[
  \OSD(\XX,\RR) \, \approx \, 37.335{\ts\ts}35 \quad \text{and} 
  \quad \nts\overline{\dim}_B(\partial W) \approx 1.946{\ts\ts}43 \ts .
\]
Indeed, inspecting the windows, compare \cite{TAO}, it is visually
hardly plausible that they are even topologically regular, as they
must, as a consequence of their dynamical origin from a window IFS.
\exend
\end{example}

\begin{example}\label{ex:ternary}
  Next, we consider a ternary unimodular Pisot inflation with three
  distinct real (and positive) eigenvalues,
\[
  a \rightarrow cab, \quad b \rightarrow ba,
  \quad c \rightarrow a.
\]
Here, $\lambda \approx 2.246{\ts\ts}980$ is the largest root of
$x^3-2x^2-x+1$, whereas $\lambda_{\dc} \approx 1.801{\ts\ts}938$ is
the largest root of $x^3-x^2-2x+1$. Interestingly, $\lambda_{\dc}$
happens to be equal to the inverse of the smallest Galois conjugate of
$\lambda$. The inverse of the other Galois conjugate is
$\lambda/\lambda_{\dc} \approx 1.246{\ts\ts}980$. For the $\OSD$, for
which the different moduli of the Galois conjugates play no role, we
now get
\[
    \OSD(\XX,\RR) \, \approx \, 3.667{\ts\ts}86 \ts .
\]
The upper bound on the window boundary dimension obtained from
\eqref{eq:dimfixpointbd} is larger than $2$, and therefore not
useful. However, a better bound adapted to the case of non-isotropic
scaling is given in \cite{FFIW06}. Combining this with the lower bound
from \eqref{eq:bddw}, we obtain the approximate bounds
\[
  1.454{\ts\ts}723 \, \leqslant \, \overline{\dim}_{B} (\partial W)
  \, \leqslant \, 1.625{\ts\ts}168 \ts ,
\]
which are reasonably sharp. In the case of non-isotropic scaling,
there currently seem to be no good methods to determine the boundary
dimension exactly.  \exend
\end{example}

\begin{example}\label{ex:constlen}
  Finally, we consider a constant-length inflation rule, which is
  necessarily non-unimodular,
\[
    a \rightarrow abab, \quad b \rightarrow caab,
    \quad c \rightarrow bcac \ts .
\]
It is rather easy to see that the discrepancy inflation matrix has the
characteristic polynomial $x^3-3x^2-x+4$, which is irreducible, with
largest root $\lambda_{\dc}\approx 2.8608$. The $\OSD$ thus becomes
$\OSD(\XX,\RR)\approx 4.1358$.  
\exend
\end{example}

\subsection{Two-dimensional tilings}

For standard cut-and-project tilings with a two-dimensional Euclidean
internal space and polygonal windows, \eqref{eq:bddw} gives
$d^{}_{\mathrm{int}}$ as an upper bound on the $\OSD$. For a Penrose
tiling with minimal embedding \cite[Ex.~7.11 and Rem.~7.8]{TAO}, this
results in a maximal value of $2$ for the $\OSD$.  We have verified
that this upper bound is indeed the exact value, and we expect an
analogous result to hold for other examples with polygonal windows
(and analogously in higher dimensions), provided a minimal embedding
is used; see \cite{Crelle} or \cite[Sec.~5.2]{Nicu} for methods to
determine such an embedding from intrisic tiling data.

Another simple way to construct two-dimensional
inflation tilings uses direct products of one-dimensional ones.
These are covered by Theorem~\ref{thm:product-system}.

\begin{example}\label{ex:fibofibotw}
  Let us consider the direct product $\XX \times \YY$ of the hulls the
  squared Fibonacci inflation $\varrho_{_\mathrm{F}}^2$ from
  Example~\ref{ex:fibo} and the reshuffled Fibonacci inflation
  $\widetilde{\varrho}$ from Example~\ref{ex:fibotwist}. The linear
  scaling factor then is $\lambda=\phi^2$ in both directions.  By
  Theorem~\ref{thm:product-system}, we obtain
\[
  \OSD(\XX {\ts\times\ts} \YY,\RR^2)
  \, = \, \OSD(\XX,\RR) + \OSD(\YY,\RR)
  \, \approx \, 12.874{\ts\ts}434 \, ,
\]
even though we have two different Lyapunov exponents, neither of which
gives the above value. In fact, discrepancy overlaps can be of three
kinds: coincidence $\times$ discrepancy, discrepancy $\times$
coincidence, and discrepancy $\times$ discrepancy.  The first two
kinds transform among themselves, having their own Lyapunov exponents,
\[
    \log\frac{\phi^2\cdot (1+\sqrt{2}\, )}{\phi^4}
    \, = \, \log(1+\sqrt{2}\,) - 2 \log(\phi)
    \qquad\text{and}\qquad
    \log\frac{1\cdot\phi^2}{\phi^4} \, = \, - 2 \log(\phi) \ts ,
\]
which are just the ones inherited from the two factors.  Discrepancies
of the third kind contribute to all three kinds, so that the larger of
the two Lyapunov exponents above applies to them.
  
Since the window of a product tile is the product of the windows of
the corresponding factor tiles,
$W_{ij} = W_i^{(1)} {\times\ts} W_j^{(2)}$, its boundary is given by
\[
  \partial W^{}_{ij} \, = \,
  \bigl(\partial W_i^{(1)}  {\times\ts} W_j^{(2)}\bigr)
  \cup \bigl( W_i^{(1)} {\times\ts}\ts \partial W_j^{(2)} \bigr) ,
\]
and consequently consists of subsets of different dimensions,
\[
  d^{}_1 \, = \, \dim \bigl( \partial W_i^{(1)}\bigr)
      + \dim \bigl( W_j^{(2)} \bigr) \quad \textrm{and} \quad
  d^{}_2 \, = \, \dim \bigl( W_i^{(1)}\bigr) +
      \dim \bigl( \partial W_j^{(2)} \bigr)  .
\]
This holds in fact for all product tile types, with the same
dimensions. Hence, despite the unique value of the $\OSD$, we have
window boundary parts with two different dimensions.  \exend
\end{example}

\begin{example}\label{ex:dpvfibo}
  Our next example emerges from direct product variations of Fibonacci
  tilings \cite{BFG21}. The ones with polygonal windows all represent
  topologically conjugate dynamical systems \cite{BGM22}, and so their
  $\OSD$ and window boundary dimensions are equal to those of the
  plain direct product of two Fibonacci tilings, namely
  $\OSD(\XX,\RR^2)=2$ and $\overline{\dim}_B(\partial W)=1$.  More
  interesting are the three types of systems with fractally bounded
  windows, which in \cite{BFG21} were called \emph{castle},
  \emph{cross} and \emph{island}. For each type, the discrepancy
  inflation is primitive, with its leading eigenvalue being the
  largest root of the polynomials
\[
  p(x) \, = \,
  \begin{cases}
     x^3-4x^2+5x-3 \ts ,                   & \text{castle} \ts , \\
     x^9-2x^8-x^7+x^6+x^5-4x^4-2x^3-1 \ts , & \text{cross} \ts ,  \\
     x^5-2x^4-x^3+2x^2+x-4 \ts ,           & \text{island} \ts . 
  \end{cases}
\]
Note that, for the second polynomial, there is a deviation from
\cite{BFG21}, although the largest roots are very close
numerically. This deviation still needs to be resolved. Numerically,
the largest roots $\lambda_{\dc}$ of the polynomials are,
approximately, $2.46557$, $2.33157$, and $2.11978$. For the $\OSD$, we
then get
\[
  \OSD(\XX,\RR^2) \, = \,
  \frac{2 \log(\phi)}{2 \log(\phi)-\log(\lambda_{\dc})}
  \, \approx \,
  \begin{cases}
    16.040 \ts , & \text{castle} \ts , \\
     8.305 \ts , & \text{cross} \ts ,  \\
     4.559 \ts , & \text{island} \ts . 
  \end{cases}
\]
  Correspondingly, the dimensions of the window boundaries become
\[
  \overline{\dim}_B (\partial W) \, = \,
  \frac{\log(\lambda_{\dc})}{\log(\phi)} \, \approx \,
  \begin{cases}
    1.8753 \ts , & \text{castle} \ts , \\
    1.7592 \ts , & \text{cross} \ts ,  \\
    1.5613 \ts , & \text{island} \ts . 
  \end{cases}
\]
Thus, the $\OSD$ distinguishes these classes topologically, as stated
in \cite{BFG21} on the basis of the window boundary dimensions.
\exend
\end{example}

\begin{example}\label{ex:hat}
  The Hat tiling \cite{Hat} has recently attracted a lot of attention
  due to its aperiodic monotile property, but it is interesting also
  for its $\OSD$. We find two different Lyapunov exponents,
\[
    \lambda_{\mathrm{L}}^{(1)} \, = \, \log \frac{2+\sqrt3}{\phi^4}
    \quad \text{and} \quad
    \lambda_{\mathrm{L}}^{(2)} \, = \, \log \frac{\phi^2}{\phi^4}
    \, = \, - 2 \log (\phi) \ts ,
\]
where $\phi$ is again the golden mean. These two Lyapunov exponents
can also be seen in the window boundary, which has straight parts and
fractal parts \cite{BGS23}. At first sight, this gives only upper and
lower bounds on the $\OSD$. However, the Hat tiling has a cousin, the
Golden Hex Tiling (GHT) of \cite{AA23}, whose construction indicates
that it is topologically conjugate to the Hat tiling.  Indeed, we find
the same two Lyapunov exponents also for the GHT. Here, the dynamical
system of the GHT has a factor, which is obtained by replacing the two
kinds of trapezoids by a parallelogram and a triangle each. This
simplified GHT has only \emph{one} Lyapunov exponent, the larger of
the above two.

Since the $\OSD$ cannot increase under a factor map, the full GHT must
have the same $\OSD$ as the simplified one, and if the Hat tiling is
conjugate to the GHT, the same holds for the Hat tiling. As a result,
we have
\[
    \OSD( \XX^{}_{\mathrm{Hat}}, \RR^2 ) \, = \,
    \frac{4\log(\phi)}{4\log(\phi) - \log (2+\sqrt{3} \,)}
    \, \approx \, 3.166{\ts\ts}443 \ts .
\]
As shown in \cite{BGS23}, the window boundaries consist of components
with dimensions
\[
  d^{}_1 \, = \, 1 \quad \textrm{or} \quad
  d^{}_2 \, = \, \frac{ \log(2+\sqrt{3} \, ) }{2\log(\phi)} \ts ,
\]
which correspond to the two Lyapunov exponents of the Hat tiling.  We
see here that, despite the two different Lyapunov exponents, we can
have a well-defined $\OSD$, whereas the window boundary contains
components of different dimensions; compare also
Example~\ref{ex:fibofibotw}.  \exend
\end{example}

\section{Concluding remarks}\label{sec:conc}

Barge and Diamond \cite{BD07} have encoded the proximality structure
of one-dimensional inflation tilings with pure-point spectrum by the
corresponding balanced pair discrepancy inflations,
$\varrho^{}_{\dc}$.  They have shown (under an additional no-cycle
condition, which is believed to be only technical) that two tiling
spaces are homeomorphic if and only if the tiling spaces generated by
their respective $\varrho^{}_{\dc}$ are homeomorphic. In
\cite[Ex.~6.11]{MR18}, Maloney and Rust have used this to distinguish
two tiling spaces which are hard to distinguish by any simpler
means. Concretely, they have computed bounds on the cohomology ranks
of the two spaces, and could thereby tell them apart. The $\OSD$
encodes (in sufficiently nice situations) only the
leading eigenvalue of $\varrho^{}_{\dc}$, and is thus much simpler to
compute. Nevertheless, it allows to distinguish these two spaces as
well, and therefore constitutes a powerful tool.

Above, we have assumed that our tilings are self-similar, with an
inflation scaling $\lambda$ that is isotropic. This restriction is not
essential, and, with minor adjustments, most of our results hold also
in the self-affine case, where the inflation scaling is a general,
strictly expanding linear map $Q$, whose expanding eigenvalues form a
Pisot family \cite{LS12}. Essentially, the volume scaling factor
$\lambda^d$ has to be replaced by $\lvert \det(Q) \rvert$, as is also
well known from the spectral analysis of such tilings \cite{BGM19}.
Some further complications might arise though, especially regarding
the window boundary dimensions.

We have also assumed that our inflations are \emph{stone inflations},
which means that the super-tiles are precisely dissected into copies
of the original tiles, where we consider only the case where the
super-tiles are scaled copies of the prototiles. Solomyak's overlap
algorithm \cite{Sol97} also assumes this. This is more restrictive
than necessary. It is enough if each tile is assigned to a unique
supertile. The limiting shape of a supertile may then become fractal,
which makes the overlap algorithm difficult to apply, because it may
be hard to determine whether two fractal shapes do overlap. However,
Akiyama and Lee \cite{AL11} have developed a generalised overlap
algorithm, which can cope with this situation, by considering almost
overlaps. With this overlap algorithm, the $\OSD$ is straight-forward
to compute.

Our approach can directly be applied only to self-similar (or to
self-affine) tilings. Clark and Sadun \cite{CS1} have investigated how
the spectral properties of such dynamical systems change under shape
changes of the tiles. They showed that shape changes which are
asymptotically negligible lead to dynamical systems that are
topologically conjugate to the original one, and hence do not change
the $\OSD$. Further, applying a global linear map to a tiling space is
a relevant shape change, but it will also not affect the $\OSD$. Other
asymptotically non-negligible shape changes, however, will generically
kill all (non-trivial) point spectrum, and hence change the $\OSD$
drastically (from a finite to an infinite value). Asymptotic
negligibility is controlled by the cohomology class in
$\check{H}^1(\XX,\RR^d)$ of a shape change \cite{CS1}.

In a sense, $\OSD$ and \v{C}ech cohomology \cite{AP98} of the
underlying tiling space are complementary complexity measures.  Both
depend on the set of singular fibres over the MEF.  The complexity of
\v{C}ech cohomology comes from the multiplicity of the singular
fibres, which is completely irrelevant for $\OSD$, as all tilings in a
fibre are identified. Rather, $\OSD$ depends on the geometry of the
set of singular fibres. The two complexity measures therefore see
different origins of complexity, and complement each other.

In the case where we have a Lyapunov spectrum with several values, we
may be able to say more than giving the rough bounds of
Theorem~\ref{thm:ac}, as we saw in Example~\ref{ex:fibofibotw}.  In
other examples like the Hat tiling \cite{Hat,BGS23}, the larger of two
Lyapunov exponents seems to be the relevant one. In this area, better
criteria are still needed.

\section*{Acknowledgements}

It is our pleasure to thank Maik Gr\"{o}ger for illuminating
discussions, and Henk Bruin, Natalie Frank and Karl Petersen for
helpful comments.

This work was supported by Deutsche Forschungsgemeinschaft (DFG,
German Research Foundation), via TRR 358/1 2023--491392403 (MB, FG)
and Project 509427705 (PG).

\end{document}